\documentclass[11pt,reqno]{amsart}

\usepackage[T1]{fontenc}
\usepackage[utf8]{inputenc}
\usepackage{lmodern}
\usepackage{microtype}
\usepackage[margin=1.02in]{geometry}
\usepackage{mathtools,amssymb,amsfonts}
\usepackage{booktabs,array,tabularx,float}
\usepackage{enumitem}
\usepackage{xcolor}
\usepackage{hyperref}
\usepackage[nameinlink,capitalize,noabbrev]{cleveref}

\definecolor{deepblue}{HTML}{17365D}
\definecolor{burgundy}{HTML}{7B2131}
\hypersetup{
  colorlinks=true,
  linkcolor=deepblue,
  citecolor=burgundy,
  urlcolor=deepblue,
  pdftitle={Rainbow spanning configurations in pseudorandom graphs}
}

\newtheorem{theorem}{Theorem}[section]
\newtheorem{lemma}[theorem]{Lemma}

\newtheorem{observation}[theorem]{Observation}
\theoremstyle{definition}
\newtheorem{definition}[theorem]{Definition}
\theoremstyle{remark}
\newtheorem{remark}{Remark}

\newcommand{\eps}{\varepsilon}
\newcommand{\evec}{\vec e}
\newcommand{\cA}{\mathcal A}
\newcommand{\cE}{\mathcal E}
\newcommand{\cF}{\mathcal F}
\newcommand{\cK}{\mathcal K}
\newcommand{\cT}{\mathcal T}
\newcommand{\ind}{\mathbf 1}

\setlist[enumerate]{leftmargin=2.15em,itemsep=0.45em,topsep=0.45em}
\setlist[itemize]{leftmargin=2.0em,itemsep=0.38em,topsep=0.38em}

\allowdisplaybreaks

\title[Rainbow spanning configurations]
{Rainbow spanning configurations in\\ uniformly coloured pseudorandom graphs}
\author[Aigner-Horev, Hefetz, Person, and Trushkin]{Elad Aigner-Horev, Dan Hefetz, Yury Person, and Michael Trushkin}

\begin{document}

\begin{abstract}
We prove a quantitative palette-transference principle for rainbow
spanning configurations in uniformly edge-coloured pseudorandom graphs.
The input of our transference principle is an embedding result of a spanning configuration in an appropriately bijumbled graph $J$ with sufficiently large minimum degree. The output of our transference principle is the asymptotically almost sure existence of a rainbow copy of the same configuration in a uniformly edge-coloured graph $G$ whose bijumbledness and minimum degree are only slightly better than those of $H$. 

We then apply our transference principle in order to asymptotically almost surely obtain $K_k$-factors, including perfect matchings, Hamilton cycles, and a prescribed spanning tree in bijumbled graphs with appropriate parameters. With the exception of spanning trees, the bijumbledness parameters we require in all of these results are essentially best possible. Moreover, in all of our results, the palette size exceeds the size of the target configuration by $\varepsilon n$, where $\varepsilon > 0$ is arbitrarily small yet fixed, and $n$ is the order of the configuration.   
\end{abstract}

\maketitle

\section{Introduction} 

We begin by recalling some of the many known results pertaining to the emergence of rainbow configurations in edge-coloured graphs.

\subsection{Rainbow configurations in uniformly coloured random graphs}
Write $\mathbb{G}_{Q}(n,p)$ to denote the binomial random graph
$\mathbb{G}(n,p)$ whose edges are coloured independently and
uniformly at random from a palette of $Q$ colours; this model of random graphs was first studied by Frieze and McKay~\cite{FriezeMcKay1994}. Table~\ref{tab:rainbow-random-graphs} provides a concise summary of the various results pertaining to the emergence of rainbow concrete configurations in $\mathbb{G}_Q(n,p)$. 

\medskip
\begin{center}
\begin{minipage}{\textwidth}
\centering
\refstepcounter{table}
\label{tab:rainbow-random-graphs}
\small
\textbf{Table~\thetable.} Selected benchmark results for uniformly
coloured random graphs

\medskip
\begin{tabularx}{\textwidth}{@{}
 >{\raggedright\arraybackslash}p{0.265\textwidth}
 >{\raggedright\arraybackslash}p{0.235\textwidth}
 >{\raggedright\arraybackslash}X@{}}
\toprule
Authors & Configuration & Sufficient parameters in $\mathbb{G}_{Q}(n,p)$ \\
\midrule
Frieze--McKay \cite{FriezeMcKay1994}
& Some spanning tree
& $Q=n-1$, $p\geq(2+\varepsilon)\log n/n$; in fact, a
hitting-time result \\
Bal--Frieze \cite{BalFrieze2016}
& Perfect matching
& $Q=n/2$, $p\geq C\log n/n$, for even $n$ \\
Ferber--Krivelevich \cite{FerberKrivelevich2016};
Ferber \cite{Ferber2015}
& Hamilton cycle
& $Q=(1+o(1))n$ at
$p\geq(\log n+\log\log n+\omega(1))/n$; 

alternatively, $Q=n$ at $p\geq C\log n/n$ \\
Ferber--Krivelevich \cite{FerberKrivelevich2016} and
Johansson--Kahn--Vu \cite{JohanssonKahnVu2008};
Han--Yuan \cite{HanYuan2025}
& $K_k$-factor, $k\mid n$, where
$h_k=(k-1)n/2$
& $Q=h_k+\varepsilon n$ at
$p\geq C_{k,\varepsilon}n^{-2/k}
(\log n)^{2/[k(k-1)]}$; alternatively, the exact palette
$Q=h_k$ at $p\geq C_k n^{-2/k}\log n$ \\
Bell--Frieze--Marbach \cite{BellFriezeMarbach2024};
Han--Yuan \cite{HanYuan2025}
& A prescribed spanning tree $T$ with $\Delta(T)\leq\Delta$
& $Q=n-1$, $p\geq C_{\Delta}\log n/n$ \\
\bottomrule
\end{tabularx}
\end{minipage}
\end{center}

\medskip

Ferber, Nenadov and Peter~\cite{FerberNenadovPeter2016}
were the first to prove a broad prescribed spanning-graph result. They showed that if $H$ has bounded maximum degree and its maximum average degree is at most some integer $d \geq 2$, then $Q = (1 + \alpha) e(H)$ and
$p \geq n^{-1/d} \log^{5/d}n$ suffice, where $\alpha > 0$ is arbitrarily small. The work of Ferber and Krivelevich~\cite{FerberKrivelevich2016} is of special relevance to this paper as they identified the McDiarmid-type coupling~\cite{McDiarmid1981} as a general tool for obtaining such results. They show that an uncoloured $h$-edge containment result at density $p_0$ transfers to the uniformly $Q$-coloured model at density $Q p_0/(Q-h+1)$, whenever the latter is at most 1.   

Bell, Frieze and Marbach \cite{BellFriezeMarbach2024}, and
subsequently Han and Yuan \cite{HanYuan2025}, developed general
rainbow-threshold theorems for sufficiently well-distributed families
of candidate configurations.  Their method represents the candidates
as the edges of an auxiliary hypergraph and imposes a quantitative
spread condition controlling how many candidates can contain any
prescribed set of graph edges.

\subsection{Rainbow spanning configurations in uniformly coloured randomly
perturbed graphs}

For an $n$-vertex graph $G_0$, write $(G_0 \cup \mathbb{G}(n,p))_Q$ to denote the $n$-vertex union in which every edge, whether belonging to the deterministic seed $G_0$ or to the random perturbation $\mathbb{G}(n,p)$, is coloured independently and uniformly at random from a palette consisting of $Q$ colours.  

The study of (uncoloured) randomly perturbed graphs was initiated by Bohman, Frieze, and Martin~\cite{BohmanFriezeMartin2003}; it were 
Anastos and Frieze \cite{AnastosFrieze2019} who were the first to study the emergence of rainbow configurations in uniformly
coloured randomly perturbed graphs.

Table~\ref{tab:rainbow-perturbed-graphs} provides a concise summary of some of the known results pertaining to rainbow configurations in uniformly
coloured randomly perturbed graphs. Here,
$\delta, \varepsilon > 0$, $d \geq 2$ and $s \geq 1$ are fixed, and the
constants $C$ may depend on them.

\medskip

\begin{center}
\begin{minipage}{\textwidth}
\centering
\refstepcounter{table}
\label{tab:rainbow-perturbed-graphs}
\small
\textbf{Table~\thetable.} Selected benchmark results for uniformly
coloured randomly perturbed graphs

\medskip
\begin{tabularx}{\textwidth}{@{}
 >{\raggedright\arraybackslash}p{0.215\textwidth}
 >{\raggedright\arraybackslash}p{0.215\textwidth}
 >{\raggedright\arraybackslash}p{0.125\textwidth}
 >{\raggedright\arraybackslash}p{0.135\textwidth}
 >{\raggedright\arraybackslash}X@{}}
\toprule
Authors & Configuration & Palette $Q$ & Perturbation $p$ & Deterministic seed $G_0$ \\
\midrule
Anastos--Frieze \cite{AnastosFrieze2019}
& $s$ edge-disjoint Hamilton cycles
& $Q\geq(120-20\log\delta)n$
& $p\geq C_{\delta,s}/n$
& arbitrary, with $\delta(G_0)\geq\delta n$,
  $0<\delta<1/2$ \\
Aigner-Horev--Hefetz \cite{AignerHorevHefetz2021}
& Hamilton cycle
& $Q=(1+\varepsilon)n$
& $p\geq C_{\delta,\varepsilon}/n$
& arbitrary, with $\delta(G_0)\geq\delta n$ \\
Katsamaktsis--Letzter--Sgueglia
\cite{KatsamaktsisLetzterSgueglia2024}
& Hamilton cycle
& $Q=n$
& $p\geq C_{\delta}/n$
& arbitrary, with $\delta(G_0)\geq\delta n$ \\
Aigner-Horev--Hefetz--Lahiri
\cite{AignerHorevHefetzLahiri2023}
& A prescribed spanning tree $T$, $\Delta(T)\leq d$
& $Q=(1+\varepsilon)n$
& $p=\omega(1)/n$
& arbitrary, with $\delta(G_0)\geq\delta n$ \\
Aigner-Horev--Hefetz--Lahiri
\cite{AignerHorevHefetzLahiri2023}
& Some spanning tree
& $Q=n-1$
& $p=\omega(n^{-2})$
& arbitrary, with $\delta(G_0)\geq\delta n$ \\
Katsamaktsis--Letzter--Sgueglia
\cite{KatsamaktsisLetzterSgueglia2025}
& A prescribed spanning tree $T$, $\Delta(T)\leq d$
& $Q=n-1$
& $p\geq C_{\delta,d}/n$
& arbitrary, with $\delta(G_0)\geq\delta n$ \\
\bottomrule
\end{tabularx}
\end{minipage}
\end{center}

\medskip

The prescribed tree in the fourth and sixth rows is fixed before the
random perturbation and colouring are sampled; simultaneous rainbow
universality is not asserted. Following the work of Ferber and Krivelevich~\cite{FerberKrivelevich2016}, Katsamaktsis, Letzter and Sgueglia~\cite{KatsamaktsisLetzterSgueglia2025} established a general McDiarmid-type transference principle for randomly perturbed graphs. Roughly put, it asserts that after independently thinning the deterministic seed, an uncoloured perturbed containment result transfers, with a constant factor increase in the perturbation probability, to a rainbow result with $(1 + \varepsilon) e(H)$ colours, where $\varepsilon > 0$ is arbitrarily small yet fixed.  

For the sake of brevity, the table is restricted to undirected hosts.  In the directed setting, Katsamaktsis, Letzter and Sgueglia~\cite{KatsamaktsisLetzterSgueglia2024} proved the corresponding
exact-palette Hamiltonicity result under a linear minimum semi-degree
condition; and Krueger and Staudinger~\cite{KruegerStaudinger2026}
subsequently obtained, with $n$ colours and $p=C/n$, simultaneous
rainbow copies of every orientation of every cycle of every possible
length.

\subsection{Pseudorandom graphs} For a
graph $G$ and sets $X,Y \subseteq V(G)$, define the ordered edge
count
$$
 \evec_G(X,Y) := \bigl|\{(x,y) \in X \times Y : xy \in E(G)\}\bigr|;
$$
note that edges spanned by $X\cap Y$ are counted twice. If $X$ and $Y$ are disjoint, then $e_G(X,Y) = \evec_G(X,Y)$; we thus use either notation.

The following definition of bijumbled graphs was introduced in~\cite{KoRoScSiSk2007}. 
\begin{definition}[Bijumbledness]
\label{def:jumbledness}
Let $n$ be a positive integer and let  $0 < p := p(n) \leq 1$ and $\beta := \beta(n) \geq 0$ be real numbers. An $n$-vertex graph $G$
is said to be \emph{$(p,\beta)$-bijumbled} if
$$
 \left|
  \evec_G(X,Y)-p|X||Y|
 \right|
 \leq
 \beta\sqrt{|X||Y|}
$$
for every $X,Y\subseteq V(G)$.  If
$G=(U,W;E)$ is bipartite, the same terminology means that
$$
 \left|
  e_G(X,Y)-p|X||Y|
 \right|
 \leq
 \beta\sqrt{|X||Y|}
$$
for every $X\subseteq U$ and $Y\subseteq W$.
\end{definition}

 An \emph{$(n,d,\lambda)$-graph} is an $n$-vertex $d$-regular
graph whose second largest eigenvalue in absolute value is $\lambda$. A consequence of the {\sl expander mixing lemma}~\cite{AlonChung1988} (see~\cite[Corollary~9.2.5]{AlonSpencer2016} as well) is that every such
graph is $(d/n,\lambda)$-bijumbled.  In the sequel, spectral notation is employed to report on previous results only; all of our contributions are stated for the more general class of bijumbled graphs. 

\subsection{Clique factors in pseudorandom graphs}

For an integer $k \geq 2$, a \emph{$K_k$-factor} of an $n$-vertex \emph{host} graph is a collection of vertex-disjoint copies of $K_k$, spanning all vertices of the host; in particular, a $K_2$-factor is a perfect matching. A prerequisite for the existence of such a configuration in any $n$-vertex graph is that $k \mid n$.

Table~\ref{tbl:Kk-factors} records the progression of results pertaining to the emergence of clique-factors in bijumbled graphs. The table does not account for the special case of perfect matchings. Moreover, it does not record the full formulations of the results listed; it keeps track only over the bound imposed on $\beta$ for each result, as this bound is deemed the most informative. In the first row, the host in the underlying result is an $(n,pn,\beta)$-graph; all other results are explicitly mentioned in the corresponding papers to hold for bijumbled graphs. In each row, $c>0$ is sufficiently small, with the precise dependence specified in the cited theorem.

\begin{table}[H]
\centering
\refstepcounter{table}
\small
\textbf{Table~\thetable.} Emergence of $K_k$-factors, $k \geq 3$ fixed, in bijumbled graphs
\label{tbl:Kk-factors}
\begin{tabularx}{\textwidth}{@{}
  >{\raggedright\arraybackslash}p{0.39\textwidth}
  >{\raggedright\arraybackslash}p{0.20\textwidth}
  >{\raggedright\arraybackslash}X@{}}
\toprule
Authors & Host requirements & Configuration \\
\midrule

Krivelevich--Sudakov--Szab\'o
\cite{KrivelevichSudakovSzabo2004}
&
$\displaystyle \beta\leq cp^3n/\log n$
&
$K_3$-factor
\\

Allen--B\"ottcher--H\`an--Kohayakawa--Person
\cite{AllenBottcherHanKohayakawaPerson2017}
&
$\displaystyle \beta\leq cp^{5/2}n$
&
Square of a Hamilton cycle; hence a $K_3$-factor when $3\mid n$
\\

Allen--B\"ottcher--H\`an--Kohayakawa--Person
\cite{AllenBottcherHanKohayakawaPerson2017}
&
$\displaystyle \beta\leq cp^{3k/2}n$
&
$k$th power of a Hamilton cycle; hence a $K_{k+1}$-factor when $(k+1)\mid n$
\\

Nenadov
\cite{Nenadov2019}
&
$\displaystyle \beta\leq cp^2n/\log n$
&
$K_3$-factor
\\

Han--Kohayakawa--Morris--Person
\cite{HanKohayakawaMorrisPerson2019}
&
$\displaystyle \beta\leq cp^kn$
&
$K_k$-factor,  $k\geq 3$
\\

Han--Kohayakawa--Person
\cite{HanKohayakawaPerson2021}
&
$\displaystyle \beta\leq cp^{k-1}n$
&
$K_k$-packing missing at most
$n^{1-1/(8k^4)}$ vertices 
\\

Morris
\cite{Morris2025}
&
$\displaystyle \beta\leq cp^{k-1}n$
&
$K_k$-factor, $k\geq 3$
\\

\bottomrule
\end{tabularx}
\end{table}

All upper bounds on $\beta$, seen in Table~\ref{tbl:Kk-factors}, carry with them an implicit density requirement arriving from the so-called {\sl Alon-Boppana-Spencer inequality}~\cite{ABS98}, asserting that $\beta = \Omega(\sqrt{pn})$. Throughout, we tacitly assume that $\beta \geq C \sqrt{pn}$ for some sufficiently large constant $C > 0$.

For triangle-factors, the triangle-free pseudorandom graphs constructed by Alon~\cite{Alon1994} establish the sharpness of the order $p^2 n$. For $k \geq 4$, determining whether the order $p^{k-1} n$ is sharp remains an open problem. The currently best lower bound of $\beta = \Omega(p^{k/2} n)$ stems from a recent construction of $K_k$-free pseudorandom graphs of density $\Theta\left(n^{-1/(k-1)}\right)$ due to Bishnoi, Ihringer, and Pepe~\cite{BishnoiIhringerPepe2020}.

For future reference, we conclude this section with a formulation of the aforementioned result of Morris.

\begin{theorem}[{\cite[Theorem~1.4]{Morris2025}}]
\label{thm:morris}
For every integer $k\geq3$ and every $\zeta>0$, there exists an 
$\eta : = \eta(k,\zeta) > 0$ such that any $n$-vertex
$(p,\beta)$-bijumbled graph $F$ satisfying
$$
 \delta(F) \geq \zeta p n
 \qquad \textrm{and} \qquad
 \beta\leq\eta p^{k-1}n
$$
contains a $K_k$-factor, whenever $k\mid n$.
\end{theorem}

\subsection{Bounded-degree spanning trees in pseudorandom graphs}

For integers $n > \Delta \geq 2$, let $\cT(n,\Delta)$ denote the
family of all $n$-vertex trees of maximum degree at most $\Delta$.
A graph is said to be \emph{$\cT(n,\Delta)$-universal} if it contains a copy
of every member of $\cT(n,\Delta)$.  

Table~\ref{tbl:trees} accounts for results pertaining to the emergence of spanning trees in pseudorandom graphs. As in Table~\ref{tbl:Kk-factors}, we focus on the pseudorandomness condition stipulated by the relevant results and omit additional requirements such as minimum degree conditions.  Throughout Table~\ref{tbl:trees}, $\Delta$ is assumed to be independent of $n$, and $C_\Delta > 0$ is {\sl some} constant depending solely on $\Delta$, whose value may alter between the various rows of the table. The results specified in the first, fourth, fifth, and sixth rows of Table~\ref{tbl:trees} are phrased for $(n,d,\lambda)$-graphs only; the result seen in the second row fits a certain notion of expander graphs, and the result appearing in the third row applies to bijumbled graphs. 

\begin{table}[H]
\centering
\refstepcounter{table}
\small
\textbf{Table~\thetable.} Tree-embedding results in pseudorandom graphs
\label{tbl:trees}
\begin{tabularx}{\textwidth}{@{}
  >{\raggedright\arraybackslash}p{0.31\textwidth}
  >{\raggedright\arraybackslash}p{0.35\textwidth}
  >{\raggedright\arraybackslash}X@{}}
\toprule
Authors & Host requirements & Tree embedding \\
\midrule

Alon--Krivelevich--Sudakov
\cite{AlonKrivelevichSudakov2007};
Balogh--Csaba--Pei--Samotij
\cite{BaloghCsabaPeiSamotij2010}
&
$\displaystyle \beta\leq \nu pn/\sqrt{8\Delta}$
&
Every tree of order at most $(1-\nu)n$ and maximum degree at most
$\Delta$
\\

Johannsen--Krivelevich--Samotij
\cite{JohannsenKrivelevichSamotij2013}
&
Explicit small- and large-set expansion conditions
&
$\cT(n,\Delta)$-universality
\\

Han--Yang
\cite{HanYang2022}
&
$\displaystyle
  \beta\leq pn/(2\Delta^{5\sqrt{\log n}})
$
&
$\cT(n,\Delta)$-universality
\\

Pavez-Sign\'e
\cite{PavezSigne2023}
&
$\beta\leq C_\Delta pn$
&
Every bounded-degree spanning tree with linearly many leaves
\\

Pavez-Sign\'e
\cite{PavezSigne2023}
&
$\beta\leq C_\Delta pn$
&
The square $G^2$, rather than $G$, is
$\cT(n,\Delta)$-universal
\\

Hyde--Morrison--M\"uyesser--Pavez-Sign\'e
\cite{HydeMorrisonMuyesserPavezSigne2025}
&
$\displaystyle
  \beta\leq pn/[C_\Delta \log^3 n]
$
&
$\cT(n,\Delta)$-universality
\\

\bottomrule
\end{tabularx}
\end{table}

Since the result of~\cite{HydeMorrisonMuyesserPavezSigne2025}, appearing in the last row of Table~\ref{tbl:trees}, requires the host graph to be regular, it is not suited for our needs. On the other hand, the result of Han and Yang~\cite{HanYang2022} does fit our transference principle and reads as follows.

\begin{theorem}[{\cite[Theorem~1.5]{HanYang2022}}]
\label{thm:han-yang}
Let $\Delta \geq 2$ be an integer and let $n$ be a sufficiently large integer. Then, any $n$-vertex $(p,\beta)$-bijumbled graph $F$ satisfying
$$
 \delta(F)\geq4\sqrt{\beta pn}
 \qquad \textrm{and} \qquad \beta\leq\frac{pn}{4\Delta^{5\sqrt{\log n}}},
$$
is $\cT(n,\Delta)$-universal.
\end{theorem}

\subsection{Hamilton cycles in pseudorandom graphs}
Given a graph $G = (V,E)$ and a set $X \subseteq V$, let
$$
 \Gamma_G(X) = \{v \in V \setminus X : xv \in E \textrm{ for some } x\in X\}
$$
denote the {\em external neighbourhood} of $X$ in $G$.

\begin{definition}[$C$-expander]
\label{def:C-expander}
An $n$-vertex graph $H$ is said to form a 
\emph{$C$-expander}, where $C > 0$, if both of the following assertions hold.
\begin{enumerate}[label=\textup{(\alph*)}]
\item $|\Gamma_H(X)| \geq C|X|$ for every
      $X \subseteq V(H)$ of size $|X| < n/(2C)$.
\item There is an edge of $H$ between every two disjoint sets
      $X,Y \subseteq V(H)$ of sizes $|X|,|Y|\geq n/(2C)$.
\end{enumerate}
\end{definition}

Table~\ref{tbl:hamilton} records the progression of results pertaining to the emergence of Hamilton cycles in pseudorandom graphs. As in the previous two tables, the results appear in abbreviated form.

\begin{table}[H]
\centering
\refstepcounter{table}
\small
\textbf{Table~\thetable.} Hamiltonicity results in pseudorandom graphs
\label{tbl:hamilton}
\begin{tabularx}{\textwidth}{@{}
  >{\raggedright\arraybackslash}p{0.34\textwidth}
  >{\raggedright\arraybackslash}p{0.37\textwidth}
  >{\raggedright\arraybackslash}X@{}}
\toprule
Authors & Host requirements & Hamiltonicity conclusion \\
\midrule

Krivelevich--Sudakov
\cite{KrivelevichSudakov2003}
&
$\displaystyle
  \beta\leq
  \frac{pn(\log\log n)^2}
       {1000\log n\,\log\log\log n}
$
&
Every $(n,d,\lambda)$-graph is Hamiltonian
\\

Hefetz--Krivelevich--Szab\'o
\cite{HefetzKrivelevichSzabo2009}
&
Explicit expansion hypotheses, satisfied by certain
$\log^{1-o(1)}n$-expanders
&
Every such expander is Hamiltonian
\\

Han--Kohayakawa--Morris--Person
\cite{HanKohayakawaMorrisPerson2021}
&
$\delta(G)\geq\zeta pn$ and
$\displaystyle \beta\leq\eta p^2n/\log n$
&
Every prescribed $2$-factor, including a Hamilton cycle
\\

Glock--Munh\'a Correia--Sudakov
\cite{GlockMunhaCorreiaSudakov2023}
&
$\displaystyle \beta\leq \frac{pn}{C\log^{1/3}n}$
&
Every $(n,d,\lambda)$-graph is Hamiltonian
\\

Ferber--Han--Mao--Vershynin
\cite{FerberHanMaoVershynin2024}
&
$pn\geq\log^6 n$ and
$\beta\leq pn/70000$
&
Every sufficiently large $(n,d,\lambda)$-graph is Hamiltonian
\\

Morris
\cite{Morris2025}
&
$\delta(G)\geq\zeta pn$ and
$\beta\leq\eta p^2n$
&
Every prescribed graph of maximum degree at most $2$
\\

Dragani\'c--Montgomery--Munh\'a Correia--Pokrovskiy--Sudakov
\cite{DraganicMontgomeryMunhaCorreiaPokrovskiySudakov2024}
&
$C$-expansion for a sufficiently large constant $C$;
spectrally, $\beta\leq cpn$
&
Every such expander is Hamiltonian
\\

\bottomrule
\end{tabularx}
\end{table}

The result of Dragani\'c, Montgomery, Munh\'a Correia, Pokrovskiy and Sudakov~\cite{DraganicMontgomeryMunhaCorreiaPokrovskiySudakov2024}, appearing in the last row of Table~\ref{tbl:hamilton}, facilitates our arguments in the sequel; it reads as follows.

\begin{theorem}[{\cite[Theorem~1.4]{DraganicMontgomeryMunhaCorreiaPokrovskiySudakov2024}}]
\label{thm:ham-expander}
Every $C$-expander
is Hamiltonian, whenever $C$ is a sufficiently large constant. 
\end{theorem}

\subsection{Our rainbow embedding results}

Let $\cK$ be a palette of $Q$ colours. A
\emph{$(Q,\cK)$-uniform colouring}, abbreviated to $Q$-uniform hereafter, assigns every edge of a graph $G$ a colour, chosen independently and uniformly at random from $\cK$. A subgraph of $G$ is then called   \emph{rainbow} (with respect to said colouring) provided that its edges receive pairwise distinct colours. Throughout, we write that an event holds {\sl asymptotically almost surely} (a.a.s.,\ hereafter) to denote
that the probability of said event tends to 1 as $n$ tends to infinity.

\medskip

Starting with the emergence of rainbow $K_k$-factors in uniformly coloured bijumbled graphs, note that such configurations span 
$$
 \frac nk\binom{k}{2}=\frac{k-1}{2}n
$$
edges. Our first result, stated below, essentially asserts that given an arbitrarily small yet fixed $\eps > 0$, a $\left(\tfrac{k-1}{2}+\eps\right)n$-uniform colouring of a bijumbled graph satisfying the premise of the result of Morris~\cite{Morris2025}, namely Theorem~\ref{thm:morris}, a.a.s.\  yields a rainbow $K_k$-factor. That is, an arbitrarily small linear surplus of colours, exhibited by the palette, ensures the emergence of rainbow $K_k$-factors in bijumbled graphs along uniform colourings essentially as soon as such structures (are conjectured to) first appear in bijumbled graphs.

\begin{theorem}[Rainbow clique factors]
\label{thm:rainbow-factors}
For every integer $k \geq 2$ and constants $\alpha, \eps > 0$ there exists a constant $c := c(k, \alpha, \eps) > 0$ such that the following holds. Let $Q \geq \left(\frac{k-1}{2} + \eps\right)n$ be an integer and let $G$ be an $n$-vertex $(p,\beta)$-bijumbled graph satisfying 
$$
 \delta(G)\geq \alpha p n
 \qquad \textrm {and} \qquad
 \beta\leq c p^{k-1} n,
$$
as well as $pn = \omega(\log n)$ if $k=2$. Then, a.a.s.\ \ a $Q$-uniform colouring of $G$ yields a rainbow $K_k$-factor provided that $k \mid n$. 
\end{theorem}

Our next result pertains to rainbow tree-universality with respect to  bounded-degree spanning trees. 

\begin{theorem}[Rainbow prescribed spanning trees]
\label{thm:rainbow-trees}
For every integer $\Delta \geq 2$ and constants 
$\alpha, \eps > 0$ there exists a constant 
$c := c(\alpha,\eps,\Delta)>0$ such that the following holds. Let
$G$ be an $n$-vertex $(p,\beta)$-bijumbled graph satisfying
$$
 \delta(G)\geq\alpha pn
 \qquad \textrm{and} \qquad
 \beta\leq
 c\,\frac{pn}{4\Delta^{5\sqrt{\log n}}}.
$$
If $Q \geq (1+\eps)n$, then for a $Q$-uniform colouring of $G$ it holds that 
$$
\sup_{T\in\mathcal T(n,\Delta)}
\Pr\!\left[G \textrm{ contains no rainbow copy of } T \right] = o(1).
$$
\end{theorem}

Our last rainbow embedding result targets Hamiltonicity and reads as follows. 

\begin{theorem}[Rainbow Hamilton cycles]
\label{thm:rainbow-hamilton}
For every $\alpha, \eps > 0$ there exist constants
$c \coloneqq c(\alpha,\eps)>0$ and $K \coloneqq K(\alpha,\eps)>0$ such that the
following holds.  Let $G$ be an $n$-vertex
$(p,\beta)$-bijumbled graph satisfying
$$
 \delta(G)\geq\alpha pn,
 \qquad
 \beta\leq cpn,
 \qquad \textrm{and} \qquad
 pn \geq K\log n.
$$
If $Q\geq(1+\eps)n$, then a $Q$-uniform colouring of $G$
a.a.s.\ \ yields a rainbow Hamilton cycle.
\end{theorem}

\subsubsection{Density assumptions}\label{sec:density} Our results require 
a certain lower bound on the density of the host graph. In some results, namely the case $k=2$ of Theorem~\ref{thm:rainbow-factors}, and Theorem~\ref{thm:rainbow-hamilton}, this lower bound is explicit. In others, namely the case $k \geq 3$ of Theorem~\ref{thm:rainbow-factors}, and Theorem~\ref{thm:rainbow-trees}, it is implicit. This issue is clarified in this section. We begin with the following simple observation.

\begin{observation}[The intrinsic lower bound on bijumbledness]
\label{obs:density-floor}
Let $G$ be an $n$-vertex $(p,\beta)$-bijumbled graph satisfying
$\delta(G)\geq\alpha pn$.  Then, $\beta \geq (1-p) \sqrt{\alpha pn}$.
\end{observation}

\begin{proof}
For any $v \in V(G)$, we may write 
$$
 \evec_G(\{v\},\Gamma_G(v)) = |\Gamma_G(v)| = \deg_G(v)
 \qquad \textrm{and} \qquad
 p|\{v\}||\Gamma_G(v)| = p \deg_G(v).
$$
Bijumbledness then implies 
$$
 (1-p)\deg_G(v)
 \leq
 \beta\sqrt{\deg_G(v)}.
$$
The minimum degree condition ensures that $\deg_G(v)>0$ holds. Dividing by its square root and then relying on the assumed lower bound $\deg_G(v) \geq \alpha pn$ completes the proof.
\end{proof}

For $k \geq 3$, combining Observation~\ref{obs:density-floor} with the condition $\beta \leq cp^{k-1}n$ imposed in Theorem~\ref{thm:rainbow-factors} shows that $p > 1/2$, or
$$
 pn\geq
 \left(\frac{\sqrt\alpha}{2c}\right)^{2/(2k-3)}
 n^{(2k-4)/(2k-3)}.
$$
In either case, $pn = \omega(\log n)$ holds. Proceeding to spanning trees, set $L_n = \Delta^{5\sqrt{\log n}}$ and note that Observation~\ref{obs:density-floor} and the pseudorandomness assumption imposed in Theorem~\ref{thm:rainbow-trees} collectively assert that if $p \leq 1/2$, then 
$$
 pn \geq \frac{4\alpha L_n^2}{c^2} = \omega(\log n).
$$
holds.

\medskip

By contrast, a linear upper bound $\beta \leq cpn$, as imposed in our prefect matching and Hamiltonicity results, only forces
$pn = \Omega(1)$. All of our results, however, rely on constant-rate percolation of the host graph, as part of the palette transference principle, namely Theorem~\ref{thm:transfer}. Carrying out such a percolation on a graph having many vertices of degree less than $c \log n$, for an appropriate constant $c$, leads to isolated vertices being formed a.a.s.\  post-percolation, rendering the resulting spanning configuration unrealisable.

\subsection{Palette transference principle: our main result} 

Let $V$ be a fixed $n$-element set.  An \emph{$h$-edge
configuration family} on $V$ is a family
$$
 \cA\subseteq\binom{\binom V2}{h}.
$$
Examples of such configuration families include, for instance, all perfect matchings of the complete graph with vertex-set $V$ (i.e., $h = |V|/2$), all Hamilton cycles in the complete graph with vertex-set $V$ (i.e., $h = |V|$), and so on. 

A graph $J$ is said to {\em support} a configuration family $\mathcal A$ provided that it contains a member of $\mathcal A$ as a subgraph. In a coloured graph, such an occurrence is termed \emph{$S$-rainbow} if its colours are pairwise
distinct and all belong to a (sub)palette $S$.  

\medskip
Our main result, from which Theorems~\ref{thm:rainbow-factors},~\ref{thm:rainbow-trees}, and~\ref{thm:rainbow-hamilton} are all derived, reads as follows.

\begin{theorem}[Palette-transference principle]
\label{thm:transfer}
For all positive constants $\alpha, a$, and  $K_0$, there exist constants
$$
 c_1^\circ \coloneqq c_1^\circ(\alpha) > 0,
 \qquad
 c_2^\circ \coloneqq c_2^\circ(\alpha,a) > 0,
 \qquad
 d_\star \coloneqq d_\star(\alpha) > 0,
 \qquad \textrm{and} \qquad
 C_{K_0} > 0
$$
with the following property. Given $n \geq 2$, set $d = pn$, and
suppose that $d\geq d_\star$ holds. Let $G = (V,E)$ be an $n$-vertex
$(p,\beta)$-bijumbled graph satisfying
$$
 \delta(G)\geq\alpha d
 \qquad \textrm{and} \qquad
 \beta\leq ad.
$$
Let $\cA$ be an $h$-edge configuration family on $V$, where
$2 \leq h \leq Q$. For a $(Q, \cK)$-uniform colouring of $G$, let $S \subseteq \cK$ have size $\ell \geq h$, and set
$$
 \rho = \frac{\ell-h+1}{Q}
 \qquad \textrm{and} \qquad
 \widehat\beta =
 5\rho\beta
 + C_{K_0}\left(
    \sqrt{\rho c_2^\circ d}+\sqrt{\log n}
  \right).
$$
Suppose, further, that a number $b_\star\geq0$ satisfies
\begin{equation}
 \widehat\beta\leq b_\star,
 \label{eq:transfer-interface}
\end{equation}
and that every $n$-vertex $(\rho p,b_\star)$-bijumbled graph $J$ with vertex-set $V$ satisfying 
\begin{equation}
 \delta(J)\geq\frac{c_1^\circ}{2}\rho d
 \label{eq:transfer-deterministic-degree}
\end{equation}
supports $\cA$.  Then,
\begin{align}
 \Pr\!\left[
  G\text{ contains an $S$-rainbow member of $\cA$}
 \right]
 \geq
 1-(2n)^{-K_0}
 -n\cdot\exp\left(-\frac{c_1^\circ\rho d}{8}\right).
 \label{eq:transfer-conclusion}
\end{align}
\end{theorem}

\begin{remark}
In applications of Theorem~\ref{thm:transfer}, the parameters $\rho$ and $\widehat \beta$ are referred to as its {\em retention probability} and its {\em degraded bijumbledness parameter}, respectively.     
\end{remark}

Throughout, we refer to the deterministic assumption imposed in Theorem~\ref{thm:transfer} asserting that $J$ (with certain properties) supports $\mathcal A$, as the {\em deterministic black-box assumption} of the palette-transference principle. 

The allowed sub-palette $S$, defined in Theorem~\ref{thm:transfer}, captures the surplus of colours available in the uniform colouring. Admittedly, in our proofs of Theorems~\ref{thm:rainbow-factors}, ~\ref{thm:rainbow-trees}, and~\ref{thm:rainbow-hamilton} this set always coincides with the entire palette, namely $\mathcal K$, per each theorem. Thus, for our needs in the present paper, the set $S$ can be suppressed in the formulation of Theorem~\ref{thm:transfer}. Nevertheless, we retain this set as the current formulation is conducive for absorption arguments which might be of use for future applications. 

The retention probability parameter $\rho$ stems from an implementation of {\sl McDiarmid's coupling method}~\cite{McDiarmid1981} in our proof of Theorem~\ref{thm:transfer}. This implementation leads us to consider percolations $G_\rho$ of the host graph $G$ in which each edge of the latter is retained with probability $\rho$, independently from the rest. An  application of Theorem~\ref{thm:transfer} with $S = \cK$ yields
$$
 \rho=\frac{Q-h+1}{Q}.
$$
If $h\leq\tau n$ and
$Q\geq h+\eps n$, then
\begin{equation}
 \rho
 \geq
 \frac{\eps}{\tau+\eps}.
 \label{eq:surplus-rate}
\end{equation}
Thus, a linear surplus of colours translates into a
constant-rate percolation in our proof.  Note crucially that the
right-hand side of~\eqref{eq:transfer-conclusion} tends to one whenever $\rho pn/\log n\to\infty$.

\subsection{Architecture of our proofs}\label{sec:architecture}

As stated above, Theorems~\ref{thm:rainbow-factors}, ~\ref{thm:rainbow-trees}, and~\ref{thm:rainbow-hamilton} are all derived from our so-called palette-transference principle announced in Theorem~\ref{thm:transfer}. In  this section we outline the proof methodology of Theorem~\ref{thm:transfer}. 

Three ingredients comprise our proof of Theorem~\ref{thm:transfer}. The first is called {\em palette coupling} and is captured in Lemma~\ref{lem:coupling} stated below. Roughly put, this lemma implements the coupling method associated with McDiarmid~\cite{McDiarmid1981}. It exploits the colour surplus, present in the palette of the uniform  colouring of a (not necessarily pseudorandom) host graph $F$, in order to lower bound the probability of emergence of a rainbow (spanning) configuration, using the probability that the same configuration emerges in a constant-rate percolation of $F$.  

As mentioned above, such coupling results were established in the past for truly random host graphs by Ferber and Krivelevich~\cite{FerberKrivelevich2016}, and later by Katsamaktsis, Letzter, and Sgueglia~\cite{KatsamaktsisLetzterSgueglia2025} for randomly perturbed host graphs. Adapting the argument of Ferber and Krivelevich~\cite{FerberKrivelevich2016}, our coupling lemma, namely Lemma~\ref{lem:coupling}, supports deterministic graph hosts.

\medskip
While the first ingredient, namely the palette coupling lemma, is independent of pseudorandomness, the next two ingredients rely on it; it is these two ingredients that elevate the aforementioned coupling lemma to a palette-transference principle for pseudorandom graphs. For a pseudorandom host graph $F$, Lemma~\ref{lem:percolation} estimates the degradation in pseudorandomness (measured via the parameter $\beta$) incurred by constant rate percolation. This is done through a random matrix concentration argument powered by a recent result of Bandeira and van Handel~\cite{BandeiraVanHandel2016}, stated in Theorem~\ref{thm:bvh-percolation} below. As alluded to in the definition of $\widehat \beta$ in Theorem~\ref{thm:transfer}, this degradation estimation includes a term of the form $\sqrt{\rho \Delta(F)}$; if $\rho$ is fixed and $\Delta(F)$ is somewhat large, which may be the case for a pseudorandom host, this estimation is useless. 

Anticipating this obstacle for pseudorandom hosts, the proof of Theorem~\ref{thm:transfer} starts by passing to a sufficiently pseudorandom subgraph $F^\circ \subseteq F$ in which all degrees are of the order $\Theta(p n)$; this renders the aforementioned pseudorandomness degradation estimation useful. This constitutes the third ingredient of our proof, namely Lemma~\ref{lem:regularisation}. The palette coupling, i.e. Lemma~\ref{lem:coupling}, as well as the analysis of the pseudorandomness level of the percolated output, i.e. Lemma~\ref{lem:percolation}, are both performed over $F^\circ$ and $F^\circ_\rho$, respectively, rather than over the original host.

\subsection{Organisation} Section~\ref{sec:palette-coupling} is dedicated to our palette coupling lemma, namely Lemma~\ref{lem:coupling}. In Section~\ref{sec:transfer-proof} we prove our palette-transference principle, namely Theorem~\ref{thm:transfer}. We then proceed to prove our applications: Theorem~\ref{thm:rainbow-factors} in Section~\ref{sec:rainbow-factors}, Theorem~\ref{thm:rainbow-trees} in Section~\ref{sec:rainbow-trees}, and Theorem~\ref{thm:rainbow-hamilton} in Section~\ref{sec:rainbow-hamilton}. Finally, we consider some directions for future research in Section~\ref{sec:concluding}.

\section{Palette coupling}\label{sec:palette-coupling}

The main result of this section reads as follows. 

\begin{lemma}[Palette coupling]
\label{lem:coupling}
Let $F$ be a prescribed graph and let
$\cF\subseteq\binom{E(F)}h$ be a family of $h$-edge sets. For a $(Q, \cK)$-uniform colouring of $F$, let $S \subseteq \cK$ have size $\ell \geq h$, and set
\begin{equation}
 \rho = \frac{\ell-h+1}{Q}.
 \label{eq:coupling-rate}
\end{equation}
Then,
\begin{align}
\Pr\!\left[
  \text{some $A\in\cF$ is rainbow with all colours in $S$}
 \right]
 \geq
 \Pr\!\left[F_\rho\text{ contains some $A\in\cF$}\right].
 \label{eq:coupling-conclusion}
\end{align}
\end{lemma}

Prior to proving Lemma~\ref{lem:coupling} we wish to better understand the quantitative nature of the retention probability $\rho$. Consider the realisation of a specific target $h$-edge, say $A := \{e_1,\ldots,e_h\} \in \mathcal F$, and suppose that $h-1$ of the members of $A$ have already been coloured in a rainbow fashion, where all colours were taken from $S$. For the colouration of the last edge of $A$, there are $\ell - h + 1$ available colours to choose from and the colour of the last edge is drawn uniformly from $\mathcal K$. Hence, 
$$
\Pr[\textrm{The last edge is properly coloured}] = \frac{\ell - h +1}{Q} = \rho. 
$$

\medskip
The rest of this section is dedicated to the proof of Lemma~\ref{lem:coupling}. 

\begin{proof}[Proof of Lemma~\ref{lem:coupling}]

Let $(e_1, \ldots, e_N)$ be an arbitrary enumeration of the edges of $F$.  For every $i \in \{0,\ldots,N\}$, assign a random list $L_i(e_j)$ to each
edge of $F$, independently, as follows:
\begin{itemize}
\item [(a)] if $j \leq i$, choose $\xi_j$ uniformly at random from $\cK$ and
      set $L_i(e_j) = \{\xi_j\}$;
\item [(b)] if $j > i$, set $L_i(e_j) = S$ with probability $\rho$,
      and set $L_i(e_j) = \varnothing$ with probability
      $1-\rho$.
\end{itemize}
For every $0 \leq i \leq N$, let the $i$th hybrid be the \emph{list-decorated} graph
$$
 \Gamma_i = \bigl(F, (L_i(e))_{e\in E(F)}\bigr).
$$

Call $A\in\cF$ \emph{realisable in $\Gamma_i$} if there exists
an injection $\phi : A \to S$ such that
\begin{equation}
 \phi(e) \in L_i(e) \textrm{ holds for every } e \in A.
 \label{eq:realisable}
\end{equation}
At $i=0$, the list of every edge is either $S$ or empty.  A set
$A\in\cF$ is then realisable precisely when every edge of $A$ has list
$S$: necessity follows from \eqref{eq:realisable}, while sufficiency
follows since $|A| = h \leq \ell = |S|$ ensures an injection from $A$ to $S$. Thus, the realisable members of $\cF$ in $\Gamma_0$
are exactly those contained in $F_\rho$.

At $i=N$, every list is the singleton containing the random
colour chosen for the corresponding edge. Condition \eqref{eq:realisable} then holds precisely when the
edges of $A$ have pairwise distinct colours, each
belonging to $S$. For $0 \leq i \leq N$, let
$$
 \cE_i =
 \{\Gamma_i\text{ contains a realisable member of $\cF$}\}.
$$
The above analysis of $\Gamma_0$ and $\Gamma_N$ implies that it suffices to prove that
\begin{equation}
 \Pr[\cE_i] \geq \Pr[\cE_{i-1}] \textrm{ holds for every } 1 \leq i \leq N.
 \label{eq:coupling-one-step}
\end{equation}

Fix some $1 \leq i \leq n$.  In $\Gamma_{i-1}$ and $\Gamma_i$, the lists of all edges other than $e_i$ have the same joint distribution: an edge
$e_j$ with $j<i$ has a uniformly chosen singleton list, whereas an
edge $e_j$ with $j>i$ has list $S$ with probability $\rho$
and the empty list with probability $1 - \rho$. Fix one possible configuration $L$ of these lists. For $T \subseteq \cK$, define $a(L,T)\in\{0,1\}$
to be $1$ when completing $L$ by assigning list $T$ to $e_i$
makes some member of $\cF$ realisable, and define it to be $0$
otherwise.  Conditional on this fixed $L$, the two corresponding success
probabilities are
\begin{align}
 q_{i-1}(L)
 &:= 
 \rho a(L,S)+(1-\rho)a(L,\varnothing),
 \label{eq:q-old}\\
 q_i(L)
 &:= 
 \frac1Q\sum_{s\in\cK}a(L,\{s\}).
 \label{eq:q-new}
\end{align}
Our aim is to prove that $q_i(L)\geq q_{i-1}(L)$; we distinguish between the following three cases.

\medskip
\noindent
\textbf{Case 1: $\mathbf{a(L, \varnothing) = 1.}$}
In this case, for every $s \in \cK$ we have
$$
a(L,S) = a(L,\{s\}) = 1.
$$
Equations \eqref{eq:q-old} and \eqref{eq:q-new} then give
$q_{i-1}(L) = q_i(L) = 1$.

\medskip
\noindent
\textbf{Case 2: $\mathbf{a(L, S) = 0}$.}
Clearly, $a(L,\varnothing) = 0$ holds as well. Similarly, $a(L, \{s\}) = 0$ for every $s \in S$. Finally, by definition, $a(L, \{s\}) = 0$ for every $s \in \cK \setminus S$. We conclude that $q_{i-1}(L) = q_i(L) = 0$ holds in this case.

\medskip
\noindent
\textbf{Case 3: $\mathbf{a(L,\varnothing) = 0}$ and $\mathbf{a(L,S) = 1}$.} Observe that $q_{i-1}(L) = \rho$ holds by Equation~\eqref{eq:q-old}. Choose a realisable $A^\star\in\cF$ and an injection
$\phi^\star:A^\star\to S$ witnessing $a(L,S)=1$. Note that 
$e_i\in A^\star$ must hold, since otherwise the same witness would show
$a(L,\varnothing)=1$.  Set
$$
 T_i(L) = S \setminus
 \{\phi^\star(e) : e \in A^\star \setminus \{e_i\}\}.
$$
The injection $\phi^\star$ uses exactly $h-1$ distinct colours on
the other edges entailing
\begin{equation}
 |T_i(L)|=\ell-h+1.
 \label{eq:pivotal-colours}
\end{equation}
Fix an arbitrary $s \in T_i(L)$ and define a map $\phi_s : A^\star \to S$ by setting $\phi_s(e_i) = s$ and $\phi_s(e) = \phi^\star(e)$ for every $e \in A^\star \setminus \{e_i\}$. The map $\phi_s$ is injective and it satisfies all list constraints after the list on $e_i$ is changed to
$\{s\}$. Since $s \in T_i(L)$ is arbitrary, it follows that $a(L, \{s\}) = 1$ holds for every $s \in T_i(L)$.
Using \eqref{eq:q-new}, \eqref{eq:pivotal-colours}, and~\eqref{eq:coupling-rate}, we obtain
$$
 q_i(L)
 \geq
 \frac{|T_i(L)|}{Q}
 =
 \frac{\ell-h+1}{Q}
 =\rho
 =q_{i-1}(L).
$$

The pointwise comparison holds for every possible list configuration
$L$. The law of total probability then proves~\eqref{eq:coupling-one-step}; as noted above, this concludes the proof of the lemma.  
\end{proof}

\section{Proof of the palette-transfer principle}
\label{sec:transfer-proof}

In this section, we prove Theorem~\ref{thm:transfer} along the conceptual plan divulged in Section~\ref{sec:architecture}.

\subsection{Proof of Theorem~\ref{thm:transfer}}

Having stated and proved Lemma~\ref{lem:coupling} in Section~\ref{sec:palette-coupling}, we proceed to state the two remaining ingredients in the proof of Theorem~\ref{thm:transfer}.

\begin{lemma}[Degree regularisation]
\label{lem:regularisation}
Given constants \(c, a > 0\), there exist constants  
\(c_1^{\circ} := c_1^{\circ}(c) > 0\),
\(c_2^{\circ} := c_2^{\circ}(c,a) > 0\), and 
\(d_{\star} := d_{\star}(c) > 0\) with the following property. For every positive integer \(n\), every \(0 < p := p(n) \leq 1\), and every \(\beta := \beta(n) \geq 0\), set \(d = p n\) and suppose that \(d \geq d_{\star}\). Then, every \(n\)-vertex
\((p,\beta)\)-bijumbled graph \(G\) satisfying
\[
 \delta(G) \geq c d \quad \textrm{and} \quad 
 \beta \leq a d,
\]
has a spanning subgraph \(G^{\circ}\) that is \((p,5\beta)\)-bijumbled and such that
\[
 c_1^{\circ}d
 \leq
 \delta(G^{\circ})
 \leq
 \Delta(G^{\circ})
 \leq
 c_2^{\circ}d.
\]
\end{lemma}

\begin{lemma}[Pseudorandomness degradation post percolation]
\label{lem:percolation}
Let $K_0 > 0$ be a constant and let $F$ be an $n$-vertex
$(p,\beta)$-bijumbled graph. Then, there exists a constant
$C_{K_0} > 0$, depending only on $K_0$, such that for every
$0<\rho<1$,
\begin{align*}
 &\Pr\!\left[
  F_\rho\text{ is }\left(
   \rho p,
   \rho\beta+C_{K_0}\bigl(
     \sqrt{\rho\Delta(F)}+\sqrt{\log n}
   \bigr)
  \right)\text{-bijumbled}
 \right]
 \geq1-(2n)^{-K_0}.
\end{align*}
\end{lemma}

\medskip

The proofs of Lemmas~\ref{lem:regularisation} and~\ref{lem:percolation} are delegated to Sections~\ref{sec:regularisation} and~\ref{sec:percolation}, respectively. The rest of this section is dedicated to the derivation of Theorem~\ref{thm:transfer} from Lemmas~\ref{lem:coupling}, ~\ref{lem:regularisation}, and~\ref{lem:percolation}. 

\begin{proof}[Proof of Theorem~\ref{thm:transfer}]
Prior to exposing the colour of any edge, apply Lemma~\ref{lem:regularisation} with $(c,a) = (\alpha,a)$, and fix the
resulting spanning subgraph $G^\circ\subseteq G$. Then,
\begin{equation}
 G^\circ\text{ is $(p,5\beta)$-bijumbled}\; \text{and satisfies}\;
 c_1^\circ d\leq\delta(G^\circ)
 \leq\Delta(G^\circ)\leq c_2^\circ d.
 \label{eq:transfer-core}
\end{equation}

Set
$$
 \cA^\circ
 :=
 \{A\in\cA:A\subseteq E(G^\circ)\}.
$$

For $Q, \cK$, and $S$ as per Theorem~\ref{thm:transfer}, apply Lemma~\ref{lem:coupling} to $F=G^\circ$ and
$\cF=\cA^\circ$ so as to obtain 
\begin{align}
 \Pr\!\left[
  G^\circ\text{ contains an $S$-rainbow member of $\cA^\circ$}
 \right]
\geq
 \Pr\!\left[H\text{ contains a member of $\cA^\circ$}\right],
 \label{eq:transfer-coupled}
\end{align}
where $H := G^\circ_\rho$. By~\eqref{eq:transfer-core} and~\eqref{eq:transfer-interface}, applying  Lemma~\ref{lem:percolation} to $G^\circ$ yields
\begin{align} \label{eq::jumbleEstimate}
 &\Pr\!\left[
  H \textrm{ is } (\rho p, b_{\star}) \textrm{-bijumbled}
 \right]
 \geq 1 - (2n)^{-K_0}.
\end{align}

For any fixed vertex $v \in V(H)$, the random variable 
$\deg_H(v)$ is distributed binomially with mean
$\rho \deg_{G^\circ}(v)\geq\rho c_1^\circ d$. It thus follows by Chernoff's inequality that
$$
 \Pr\!\left[
  \deg_H(v) < \frac{c_1^\circ}{2}\rho d
 \right]
 \leq
 \exp\left(-\frac{c_1^\circ\rho d}{8}\right).
$$
A union-bound over $V(H)$ establishes that
\begin{align} \label{eq::degreeEstimate}
\Pr\!\left[\delta(H) < \frac{c_1^\circ}{2}\rho d \right] \leq  n \cdot \exp\left(-\frac{c_1^\circ\rho d}{8}\right).
\end{align}

Combining the two failure probability estimates~\eqref{eq::jumbleEstimate} and~\eqref{eq::degreeEstimate} with~\eqref{eq:transfer-coupled} and the premise of Theorem~\ref{thm:transfer}, yields~\eqref{eq:transfer-conclusion}, thus concluding the proof.
\end{proof}

\subsection{Degree regularisation: proof of Lemma~\ref{lem:regularisation}}\label{sec:regularisation} 

Lemma~\ref{lem:regularisation} is deduced from the following more general  result. 

\begin{lemma}
\label{lem:quantitative-degree-regularisation}
Given $c,a>0$ as well as $L>1$, let $\kappa,\gamma>0$
satisfying
\begin{equation}
 \kappa+a\sqrt{\kappa}+\gamma\leq c
 \label{eq:quantitative-regularisation-admissibility}
\end{equation}
be set.  Then, for every positive integer $n$, every $0 < p := p(n) \leq1 $,
and every $\beta := \beta(n) \geq 0$, every $n$-vertex
$(p,\beta)$-bijumbled graph $G$ satisfying
$$
 \delta(G)\geq cd
 \qquad\text{and}\qquad
 \beta\leq ad,
$$
where $d:=pn$, contains a spanning subgraph $G^\circ\subseteq G$
satisfying
\begin{equation}
 \delta(G^\circ)\geq r:=\lfloor\gamma d\rfloor
 \qquad\text{and}\qquad
 \Delta(G^\circ)
 \leq
 \max\left\{
  Ld,
  r+\left\lceil\frac r\kappa\right\rceil
 \right\},
 \label{eq:quantitative-regularisation-degree-bounds}
\end{equation}
and $G^\circ$ is
\begin{equation}
 \left(
  p,
  \left(3+\frac{2}{L-1}\right)\beta
 \right)\text{-bijumbled}.
 \label{eq:quantitative-regularisation-jumbledness}
\end{equation}
\end{lemma}

\medskip 

Prior to proving Lemma~\ref{lem:quantitative-degree-regularisation}, we use it to deduce Lemma~\ref{lem:regularisation}.

\begin{proof}[Proof of Lemma~\ref{lem:regularisation}]
Set
\begin{equation}
 \kappa
 =
 \min\left\{
  \frac14,
  \frac c4,
  \frac{c^2}{16a^2}
 \right\}
 \qquad\text{and}\qquad
 \gamma=\frac c4,
 \label{eq:regularisation-corollary-parameters}
\end{equation}
as well as
\begin{equation}
 c_1^\circ=\frac c8
 \qquad\text{and}\qquad
 c_2^\circ
 =
 \max\left\{
  2,
  1+\frac c4\left(1+\frac1\kappa\right)
 \right\}.
 \label{eq:regularisation-corollary-degree-constants}
\end{equation}
Lastly, set
$$
 d_\star=\max\left\{1,\frac4c\right\}.
$$
It is easy to verify that
$\kappa+a\sqrt{\kappa}+\gamma\leq3c/4<c$, implying that the above
choice of parameters complies with the quantification of
Lemma~\ref{lem:quantitative-degree-regularisation}.  Apply
Lemma~\ref{lem:quantitative-degree-regularisation} with $L=2$, and note
that $d\geq d_\star\geq4/c$ implies that $cd/4\geq1$.  The
elementary inequality $\lfloor x\rfloor\geq x/2$, holding for every
$x\geq1$, then yields
\begin{equation}
 r\geq\frac{cd}{8}=c_1^\circ d.
 \label{eq:regularisation-corollary-lower-bound}
\end{equation}
Note that $L d = 2 d \leq c_2^\circ d$ holds by the choice of
$c_2^\circ$ and, moreover,
\begin{align*}
 r+\left\lceil\frac r\kappa\right\rceil
 &\leq
 r+\frac r\kappa+1
 \leq
 \frac{cd}{4}\left(1+\frac1\kappa\right)+1
 \leq
 \left[
  1+\frac c4\left(1+\frac1\kappa\right)
 \right]d
 \leq
 c_2^\circ d.
\end{align*}
Combining the above two estimates with the degree bounds seen in
\eqref{eq:quantitative-regularisation-degree-bounds}, yields the
asserted upper bound $\Delta(G^\circ)\leq c_2^\circ d$.  The lower
bound on the minimum degree appearing in
\eqref{eq:quantitative-regularisation-degree-bounds}, together with
\eqref{eq:regularisation-corollary-lower-bound}, yields
$\delta(G^\circ)\geq c_1^\circ d$.  Finally, for $L=2$, the
corresponding factor in
\eqref{eq:quantitative-regularisation-jumbledness} is
$$
 3+\frac{2}{2-1}=5.
$$
We conclude that $G^\circ$ is $(p,5\beta)$-bijumbled.
\end{proof}

The rest of this section is dedicated to the proof of Lemma~\ref{lem:quantitative-degree-regularisation}. The main idea of the proof is as follows. Let \(S\) be the
set of vertices whose degrees exceed \(Ld\), and set
\(T = V(G) \setminus S\); we distinguish $S$ as the {\em exceptional set} and refer to its elements as {\em exceptional vertices}. We begin the construction of $G^{\circ}$ by retaining all of \(G[T]\). The remaining task is to add some of the edges incident with $S$ in such a way that the required bounds on the minimum and maximum degrees are maintained. To that end, a spanning subgraph \(F\subseteq G\) such that 
\begin{equation} \label{eq::LUdegreesF}
 r \leq \deg_F(v) \leq r + \left\lceil\frac r\kappa\right\rceil
\end{equation}
hold for every $v \in V(G)$ is constructed, and we then take 
\begin{equation}
 G^{\circ} := G[T] \cup F.
 \label{eq:construction-preview}
\end{equation}

We begin by showing that the effect the specific choice of $F$ has on the bijumbledness of $G^{\circ}$ is limited.

\begin{lemma}
\label{lem:exceptional-support}
Let \(G\) be an \(n\)-vertex \((p,\beta)\)-bijumbled graph, set \(d = pn\),
fix \(L>1\), and define
\begin{equation}
 S = \{v \in V(G) : \deg_G(v) > L d\} 
 \quad \textrm{and} \quad
 T = V(G) \setminus S.
 \label{eq:exceptional-set}
\end{equation}
Then,
\begin{equation}
 |S|
 \leq
 \frac{\beta^2n}{(L-1)^2d^2}
 \label{eq:exceptional-size}
\end{equation}
holds. Moreover, if \(G[T] \subseteq H \subseteq G\) is a spanning subgraph, then \(H\) is
\[
 \left(
   p,
   \left(3+\frac{2}{L-1}\right)\beta
 \right)\text{-bijumbled}.
\]
\end{lemma}

\begin{proof}
Starting with the proof of~\eqref{eq:exceptional-size}, note that
\begin{align} \label{eq::LdS}
 L d |S| < \sum_{v \in S} \deg_G(v) = \vec e_G(S, V(G)).
\end{align}
The bijumbledness of $G$, applied to \(S\) and \(V(G)\), yields 
\begin{align} \label{eq::betaRootS}
 \vec e_G(S,V(G)) \leq p |S| n + \beta\sqrt{|S|n} = d |S| + \beta\sqrt{|S|n}.
\end{align}
If \(S=\varnothing\), then \eqref{eq:exceptional-size} is immediate.  Otherwise, combining~\eqref{eq::LdS} and~\eqref{eq::betaRootS}, subtracting
\(d|S|\) and dividing by \(\sqrt{|S|}\) yields
\[
 (L-1)d\sqrt{|S|}
 <
 \beta\sqrt n.
\]
Squaring the preceding inequality yields \eqref{eq:exceptional-size}.

\medskip
Proceeding to the second assertion of the lemma, we  quantify the discrepancy created by deleting some of the edges that are incident with $S$. Let \(J\)
be the spanning subgraph of $G$ with $E(J) := E(G) \setminus E(H)$. Since \(G[T] \subseteq H\), every edge of \(J\) has at least one
endpoint in \(S\). Therefore, for arbitrary sets \(X,Y\subseteq V(G)\), it holds that
\begin{equation}
 \vec e_J(X,Y)
 \leq
 \vec e_G(X\cap S,Y)
 +
 \vec e_G(X,Y\cap S).
 \label{eq:deleted-support}
\end{equation}
Bijumbledness applied to \(X\cap S\) and \(Y\) yields
\begin{align*}
 \vec e_G(X \cap S,Y) &\leq p |X \cap S||Y| + \beta\sqrt{|X \cap S||Y|} \leq p |X \cap S||Y| + \beta\sqrt{|X||Y|} \\
 &\leq p \sqrt{|S|n} \sqrt{|X||Y|} + \beta\sqrt{|X||Y|} \leq \left(1+\frac1{L-1}\right)
 \beta\sqrt{|X||Y|}, 
 \label{eq:first-supported-term}
\end{align*}
where the penultimate inequality holds since $|Y| \leq n$ and $|X \cap S| \leq \min \{|X|, |S|\}$, and the last inequality holds by~\eqref{eq:exceptional-size}.

An analogous argument shows that $\vec e_G(X, Y \cap S) \leq \left(1 + \frac1{L-1}\right) \beta\sqrt{|X||Y|}$. It then follows by~\eqref{eq:deleted-support} that
\begin{equation}
 \vec e_J(X,Y)
 \leq
 2\left(1+\frac1{L-1}\right)
 \beta\sqrt{|X||Y|}.
 \label{eq:deleted-discrepancy}
\end{equation} 

Finally, \(\vec e_H(X,Y)=\vec e_G(X,Y)-\vec e_J(X,Y)\), so
\begin{align*}
 \left|\vec e_H(X,Y)-p|X||Y|\right|
 &\leq
 \left|\vec e_G(X,Y)-p|X||Y|\right|
 +\vec e_J(X,Y)\\
 &\leq
 \left(3+\frac{2}{L-1}\right)
 \beta\sqrt{|X||Y|}.
\end{align*}
This holds simultaneously for all \(X,Y\subseteq V(G)\), thus concluding the proof of the lemma.
\end{proof}

Our second lemma handles the construction of $F$. Lemma~\ref{lem:exceptional-support} ensures the bijumbledness of any spanning subgraph \(G[T] \subseteq H \subseteq G\). Hence, we need only ensure the degree restrictions appearing in~\eqref{eq::LUdegreesF}.  

\begin{lemma}
\label{lem:regularisation-bounded-load}
Let $G$ be an $n$-vertex $(p,\beta)$-bijumbled graph.  Set
$d=pn$ and suppose that
$$
 \delta(G)\geq cd
 \qquad\text{and}\qquad
 \beta\leq ad,
$$
for constants $c, a > 0$.  Let $\kappa, \gamma > 0$ satisfying
\begin{equation}
 \kappa+a\sqrt{\kappa}+\gamma\leq c
 \label{eq:regularisation-bounded-load-admissibility}
\end{equation}
be fixed, and set
\begin{equation}
 r=\lfloor\gamma d\rfloor
 \qquad\text{and}\qquad
 b=\left\lceil\frac r\kappa\right\rceil.
 \label{eq:regularisation-bounded-load-parameters}
\end{equation}
Then, $G$ contains a spanning subgraph $F$ such that
\begin{equation}
 r\leq\deg_F(v)\leq r+b
 \label{eq:regularisation-bounded-load-conclusion}
\end{equation}
holds for every $v\in V(G)$.
\end{lemma}

Our proof of Lemma~\ref{lem:regularisation-bounded-load} is based on a
maximum integral flow argument.  We start by introducing the relevant
terminology and tools.  By a {\em bipartite network} we mean a quadruple consisting
of a directed graph $D=(V,E)$, a capacity function
$\operatorname{cap}:E\to\mathbb R_{\geq0}$, as well as two
distinguished vertices of $D$.  The directed graph $D$ is defined
through a directed bipartite graph, appended with a source vertex
$s$ and a sink vertex $t$ (these form the distinguished vertices).
The former is attached through outgoing arcs to all vertices in one
part of the bipartition; the latter is attached through ingoing arcs
emanating from all vertices in the other part of the bipartition.  An {\em $s$--$t$ feasible flow} is a function assigning to each edge $e\in E$ a
non-negative real value not exceeding $\operatorname{cap}(e)$ and
satisfying flow conservation (i.e., what enters a vertex exits it) at
every vertex other than $s$ and $t$.  Its value is the {\em net flow}
leaving $s$.  A feasible flow with maximum net flow value is called
a {\em maximum flow}.  An {\em $s$--$t$ cut} is a partition $(U,W)$ of the
network vertices with $s\in U$ and $t\in W$; its capacity is the
sum of the capacities of the arcs directed from $U$ to $W$.  The
following well-known result links flows and cuts.

\begin{lemma}[Integral min-cut--max-flow]
\label{lem:regularisation-integral-flow}
In a bipartite network with source $s$, sink $t$, and
non-negative integral arc capacities, the maximum value of an
$s$--$t$-feasible flow equals the minimum capacity of an
$s$--$t$ cut.  Moreover, a maximum flow can be chosen to be integral.
\end{lemma}

\begin{proof}[Proof of Lemma~\ref{lem:regularisation-bounded-load}]
The argument consists of constructing a network and then translating
a maximum integral flow in this network to the required graph $F$.

To construct the directed network $D = (V,E)$, let
$$
 V_L = \{v_L:v\in V(G)\}
 \qquad\text{and}\qquad
 V_R = \{v_R:v\in V(G)\}
$$
be two disjoint copies of $V(G)$. Let $V = V_L \cup V_R \cup \{s, t\}$.  For every edge $xy \in E(G)$, 
include both of its directed arcs in $E$, that is,
$$
 x_L y_R\in E
 \qquad\text{and}\qquad
 y_L x_R\in E.
$$
Add to $E$ all arcs in $\{s v_L : v_L \in V_L\}$ and all arcs in $\{v_R t : v_R \in V_R\}$. Set the following capacities:
$$
 \operatorname{cap}(s v_L) = r,
 \qquad
 \operatorname{cap}(x_L y_R) = 1,
 \qquad
 \operatorname{cap}(v_R t) = b.
$$
For an arbitrary $s$--$t$ cut, let $C$ denote its part
containing $s$ and define
$$
 A=\{v\in V(G):v_L\in C\}
 \qquad\text{as well as}\qquad
 B=\{v\in V(G):v_R\in C\}.
$$
The capacity of this cut is given by
\begin{equation}
 r(n-|A|)+\evec_G(A,V(G)\setminus B)+b|B|.
 \label{eq:regularisation-flow-cut-capacity}
\end{equation}
Indeed, these three terms sum the capacities of the arcs $sv_L$
with $v\notin A$, the arcs $x_Ly_R$ with $x_L\in C$ and
$y_R\in V_R\setminus C$, and the arcs $v_Rt$ with $v\in B$,
respectively.  Consequently, every $s$--$t$ cut has capacity at
least $rn$ if and only if
\begin{equation}
 \evec_G(A,V(G)\setminus B)+b|B|\geq r|A|
 \label{eq:regularisation-flow-cut-goal}
\end{equation}
holds for all $A,B\subseteq V(G)$.

To establish \eqref{eq:regularisation-flow-cut-goal}, consider the following two
complementary cases.  Suppose, first, that $|B|\geq\kappa|A|$.  Since
$b=\lceil r/\kappa\rceil\geq r/\kappa$, it follows that
$$
 b|B|\geq b\kappa|A|\geq r|A|.
$$
Since $\evec_G(A,V(G)\setminus B)\geq0$ clearly holds,
\eqref{eq:regularisation-flow-cut-goal} follows.

Suppose, then, that $|B|<\kappa|A|$; note that, in particular,
$A\neq\varnothing$. Bijumbledness and the assumed bound
$\beta \leq ad$ entail 
\begin{align}
 \evec_G(A,B)
 &\leq
 p|A||B|+\beta\sqrt{|A||B|} =
 d|A|\frac{|B|}{n} + \beta|A|\sqrt{\frac{|B|}{|A|}} \nonumber\\
 &\leq
 d|A|\kappa + ad|A|\sqrt{\frac{|B|}{|A|}}
 \leq d|A|\left(\kappa+a\sqrt\kappa\right),
 \label{eq:regularisation-flow-internal-edge-bound}
\end{align}
where the penultimate inequality holds since
$|B|/n<\kappa|A|/n\leq\kappa$, and the last inequality
holds since $\sqrt{|B|/|A|}<\sqrt\kappa$.

On the other hand, the minimum degree assumption appearing in the
premise of the lemma implies
\begin{equation}
 \evec_G(A,V(G))
 =
 \sum_{v\in A}\deg_G(v)
 \geq
 cd|A|.
 \label{eq:regularisation-flow-minimum-degree}
\end{equation}
Combining \eqref{eq:regularisation-flow-internal-edge-bound} and
\eqref{eq:regularisation-flow-minimum-degree}, and relying on
\eqref{eq:regularisation-bounded-load-admissibility} yields
$$
 \evec_G(A,V(G)\setminus B)
 =
 \evec_G(A,V(G))-\evec_G(A,B)
 \geq
 d|A|\left(c-\kappa-a\sqrt\kappa\right)
 \geq
 \gamma d|A|
 \geq
 r|A|.
$$
We conclude that \eqref{eq:regularisation-flow-cut-goal} holds in the
second case as well.

Establishing \eqref{eq:regularisation-flow-cut-goal}
means that every $s$--$t$ cut has capacity at least $rn$; in
particular, the total capacity leaving $s$ is precisely $rn$.
Lemma~\ref{lem:regularisation-integral-flow} then asserts that
$D$ admits an integral flow of value $rn$.

It remains to translate this integral flow into the graph $F$.  Let
$\mathcal M$ be the set of arcs in $E_D(V_L,V_R)$ carrying one
unit of flow.  Since the total capacity of the arcs leaving $s$
equals the flow value $rn$, each such arc is saturated.  Flow
conservation and the unit capacities of the arcs in $E_D(V_L,V_R)$
then imply that exactly $r$ arcs of $\mathcal M$ leave every
vertex $v_L$.  Similarly, at most $b$ arcs of $\mathcal M$ enter
every vertex $v_R$.

For every vertex $v\in V(G)$, let
$$
 O(v)=\{w\in V(G):(v_L,w_R)\in\mathcal M\}
 \qquad\text{and}\qquad
 I(v)=\{w\in V(G):(w_L,v_R)\in\mathcal M\}.
$$
It follows by the definitions of $\mathcal M$ and the capacity
function that
\begin{equation}
 |O(v)|=r
 \quad\text{and}\quad
 |I(v)|\leq b
 \quad\text{hold for every }v\in V(G).
 \label{eq:regularisation-flow-in-out-degrees}
\end{equation}
Let $F$ be the spanning subgraph of $G$ with edge-set
$$
 E(F)
 =
 \{vw\in E(G):(v_L,w_R)\in\mathcal M
                \text{ or }(w_L,v_R)\in\mathcal M\}.
$$
Note that
$$
 O(v)
 \subseteq
 \{w\in V(G):vw\in E(F)\}
 \subseteq
 O(v)\cup I(v)
$$
holds for every $v\in V(G)$.  Consequently, relying on
\eqref{eq:regularisation-flow-in-out-degrees}, we obtain
$$
 r=|O(v)|
 \leq
 \deg_F(v)
 \leq
 |O(v)|+|I(v)|
 \leq
 r+b,
$$
proving \eqref{eq:regularisation-bounded-load-conclusion}.
\end{proof}

\medskip
We are now in a position to prove Lemma~\ref{lem:quantitative-degree-regularisation}. 

\begin{proof}[Proof of
Lemma~\ref{lem:quantitative-degree-regularisation}]
Let $S$ and $T$ be defined by~\eqref{eq:exceptional-set}, and apply Lemma~\ref{lem:regularisation-bounded-load} to obtain a spanning subgraph
$F \subseteq G$ satisfying~\eqref{eq:regularisation-bounded-load-conclusion}. Let $G^\circ=G[T]\cup F$. Starting with the stipulated degree bounds, since $F \subseteq G^\circ$, it follows by~\eqref{eq:regularisation-bounded-load-conclusion} that $\deg_{G^\circ}(v)\geq\deg_F(v)\geq r$ holds for every $v \in V(G)$, thus proving the lower bound in
\eqref{eq:quantitative-regularisation-degree-bounds}.  If $v\in T$,
then the definition of $T$ implies that
\begin{equation}
 \deg_{G^\circ}(v)\leq\deg_G(v)\leq Ld.
 \label{eq:regularisation-ordinary-vertex-degree}
\end{equation}
Otherwise, $v\in S$ and thus
\begin{equation}
 \deg_{G^\circ}(v)
 =
 \deg_F(v)
 \leq
 r+\left\lceil\frac r\kappa\right\rceil.
 \label{eq:regularisation-exceptional-vertex-degree}
\end{equation}
Combining \eqref{eq:regularisation-ordinary-vertex-degree} and
\eqref{eq:regularisation-exceptional-vertex-degree} proves the upper
bound appearing in
\eqref{eq:quantitative-regularisation-degree-bounds}.

Proceeding to the bijumbledness of $G^\circ$, since $G^\circ$ is a
spanning subgraph of $G$ containing $G[T]$, it follows by Lemma~\ref{lem:exceptional-support}, applied with
$H=G^\circ$, that $G^\circ$ is
$$
 \left(
  p,
  \left(3+\frac{2}{L-1}\right)\beta
 \right)\text{-bijumbled}.
$$
This proves \eqref{eq:quantitative-regularisation-jumbledness} and
concludes the proof of Lemma~\ref{lem:quantitative-degree-regularisation}.
\end{proof}

\subsection{Pseudorandomness degradation post percolation: proof of Lemma~\ref{lem:percolation}}\label{sec:percolation}

Gearing up towards a proof of Lemma~\ref{lem:percolation}, we recall that the operator norm of a real $d_1\times d_2$ matrix $N$ is
given by
$$
 \|N\|_{\mathrm{op}}
 :=
 \sup_{x\in\mathbb R^{d_2}\setminus\{0\}}
 \frac{\|Nx\|_2}{\|x\|_2}.
$$
Facilitating our proof of Lemma~\ref{lem:percolation} is the following
upper-tail estimate for the operator norm of certain random matrices,
established by Bandeira and van Handel
\cite[Corollary~3.12 and Remark~3.13]{BandeiraVanHandel2016}.

\begin{theorem}[Bandeira--van Handel
  \cite{BandeiraVanHandel2016}]
\label{thm:bvh-percolation}
Let $K_0>0$ and let $N=(N_{ij})$ be a real
$d_1\times d_2$ random matrix whose entries are mutually
independent, centred, and satisfy $|N_{ij}| \leq 1$ asymptotically almost surely.
Define
\begin{equation}
 \varsigma_1^2
 :=
 \max_{1\leq i\leq d_1}
 \sum_{j=1}^{d_2}\mathbb E\!\left[N_{ij}^2\right]
 \qquad\text{and}\qquad
 \varsigma_2^2
 :=
 \max_{1\leq j\leq d_2}
 \sum_{i=1}^{d_1}\mathbb E\!\left[N_{ij}^2\right].
 \label{eq:bvh-variance-parameters}
\end{equation}
Then, there exists a constant $C_{K_0}>0$, depending only on
$K_0$, such that
\begin{equation}
 \Pr\!\left[
  \|N\|_{\mathrm{op}}
  \leq
  C_{K_0}\left(
   \varsigma_1+\varsigma_2+
   \sqrt{\log(d_1+d_2)}
  \right)
 \right]
 \geq
 1-(d_1+d_2)^{-K_0}.
 \label{eq:bvh-tail}
\end{equation}
\end{theorem}

We proceed to a proof of Lemma~\ref{lem:percolation}.

\begin{proof}[Proof of Lemma~\ref{lem:percolation}]
For every edge $e\in E(F)$, let $\xi_e$ be the indicator random
variable for the event $e\in E(F_\rho)$.  Then,
$\{\xi_e:e\in E(F)\}$ are i.i.d.\ Bernoulli random variables with
mean $\rho$.  Our proof consists of the following four steps.

\medskip
\noindent
\textbf{Step 1. Capturing all cut-discrepancies through an operator norm.}
Define the symmetric $n\times n$ matrix $M$ by setting
$$
 M_{uv}
 =
 \begin{cases}
  \xi_{uv} - \rho, & uv \in E(F) \\
  0, & uv\notin E(F) 
 \end{cases}
$$
for every (not necessarily distinct) $u,v\in V(F)$.  For arbitrary
sets $X,Y\subseteq V(F)$ we may write
$$
 \ind_X^{\mathsf T}M\ind_Y
 =
 \evec_{F_\rho}(X,Y)-\rho\evec_F(X,Y).
$$
It then follows by the triangle inequality that
\begin{align}
 \left|
  \evec_{F_\rho}(X,Y)-\rho p|X||Y|
 \right|
 &\leq
 \left|\ind_X^{\mathsf T}M\ind_Y\right|
 +
 \rho\left|
  \evec_F(X,Y)-p|X||Y|
 \right|
 \notag\\
 &\leq
 \left(\|M\|_{\mathrm{op}}+\rho\beta\right)
 \sqrt{|X||Y|},
 \label{eq:percolation-op-reduction}
\end{align}
where the last inequality holds by
$$
 \left|\ind_X^{\mathsf T}N\ind_Y\right|
 \leq
 \|N\|_{\mathrm{op}}
 \|\ind_X\|_2\|\ind_Y\|_2
 =
 \|N\|_{\mathrm{op}}\sqrt{|X||Y|}
$$
and by the $(p,\beta)$-bijumbledness of $F$.  Hence, it suffices to
prove that the inequality
\begin{equation}
 \|M\|_{\mathrm{op}}
 \leq
 C_{K_0}\left(
  \sqrt{\rho\Delta(F)}+\sqrt{\log n}
 \right)
 \label{eq:percolation-op-target}
\end{equation}
holds with probability at least $1-(2n)^{-K_0}$.

\medskip
\noindent
\textbf{Step 2. Gaining entry-wise independence.}
Fix an arbitrary ordering $v_1,\ldots,v_n$ of $V(F)$ and define
the $n\times n$ random matrix $R=(R_{ij})$ by
$$
 R_{ij}
 =
 \begin{cases}
  \xi_{v_iv_j}-\rho,
    & i<j \textrm{ and } v_i v_j \in E(F) \\
  0, & \textrm{otherwise} 
 \end{cases}
$$
so that $R$ is upper-triangular.  Each nonconstant entry of $R$
corresponds to a different edge of $F$.  Hence, all entries of $R$,
including its deterministic zero entries, are mutually independent.
These are also centred and bounded in absolute value by one.  Finally,
observe that
\begin{equation}
 M=R+R^{\mathsf T}
 \qquad\text{and therefore}\qquad
 \|M\|_{\mathrm{op}}
 \leq
 \|R\|_{\mathrm{op}}+\|R^{\mathsf T}\|_{\mathrm{op}}
 =
 2\|R\|_{\mathrm{op}}.
 \label{eq:M-R}
\end{equation}

\medskip
\noindent
\textbf{Step 3. Estimating variance parameters.}
If $i<j$ and $v_iv_j\in E(F)$, then
$$
 \mathbb E\!\left[R_{ij}^2\right]
 =
 \operatorname{Var}(\xi_{v_iv_j})
 =
 \rho(1-\rho)
 \leq\rho;
$$
the second moment of any other entry of $R$ is clearly zero.  Hence,
for every row $i$, it holds that
$$
 \sum_{j=1}^n\mathbb E\!\left[R_{ij}^2\right]
 \leq
 \rho\bigl|\{j>i:v_iv_j\in E(F)\}\bigr|
 \leq
 \rho\deg_F(v_i)
 \leq
 \rho\Delta(F).
$$
Similarly, for every column $j$, it holds that
$$
 \sum_{i=1}^n\mathbb E\!\left[R_{ij}^2\right]
 \leq
 \rho\bigl|\{i<j:v_iv_j\in E(F)\}\bigr|
 \leq
 \rho\deg_F(v_j)
 \leq
 \rho\Delta(F).
$$
Hence, the parameters appearing in
\eqref{eq:bvh-variance-parameters} satisfy
$$
 \max\{\varsigma_1,\varsigma_2\}
 \leq
 \sqrt{\rho\Delta(F)}.
$$

\medskip
\noindent
\textbf{Step 4. Conclusion.}
Applying Theorem~\ref{thm:bvh-percolation} to $R$ with $d_1=d_2=n$
yields that
\begin{equation}
 \|R\|_{\mathrm{op}}
 \leq
 C_{K_0}'\left(
  2\sqrt{\rho\Delta(F)}+\sqrt{\log(2n)}
 \right)
 \label{eq:percolation-R-tail}
\end{equation}
holds with probability at least $1-(2n)^{-K_0}$, where
$C_{K_0}'>0$ is some constant depending solely on $K_0$.

Combining \eqref{eq:M-R} and \eqref{eq:percolation-R-tail} then yields
that with probability at least $1-(2n)^{-K_0}$,
$$
 \|M\|_{\mathrm{op}}
 \leq
 2C_{K_0}'\left(
  2\sqrt{\rho\Delta(F)}+\sqrt{\log(2n)}
 \right)
 \leq
 C_{K_0}\left(
  \sqrt{\rho\Delta(F)}+\sqrt{\log n}
 \right),
$$
for an appropriate constant $C_{K_0}$.  This entails
\eqref{eq:percolation-op-target} and concludes the proof of the lemma.
\end{proof}

\section{Rainbow factors in uniformly coloured pseudorandom graphs}
\label{sec:rainbow-factors}

In this section we prove Theorem~\ref{thm:rainbow-factors}. This result has two branches, distinguished by $k \geq 3$ and $k=2$. To prove the former, we apply the palette-transference principle (Theorem~\ref{thm:transfer}) with the result of Morris~\cite{Morris2025}, namely Theorem~\ref{thm:morris}, serving as its deterministic black-box. In the remaining case $k=2$, corresponding to a perfect matching, we first pass to 
a balanced bipartite core, then apply the palette-transference principle to this core, with the deterministic black-box being a stand-alone verification of Hall's condition in an appropriate bipartite graph. This separated treatment of the two branches is necessary since the result of Morris does not cover the case $k=2$.

\subsection{Rainbow \texorpdfstring{$\mathbf{K_k}$-factors, $\mathbf{k \geq 3}$}{Kk-factors, k at least 3}}

\begin{proof}[Proof of Theorem~\ref{thm:rainbow-factors} for $k \geq 3$]
Fix $k \geq 3$, suppose that $k \mid n$, and set
$$
 \rho_\star = \frac{2\eps}{k-1+2\eps}.
$$
Aiming to apply Theorem~\ref{thm:transfer} with $a=1$, $K_0 = 2$, and $\alpha$ as per Theorem~\ref{thm:rainbow-factors}, let $c_1^\circ, c_2^\circ, d_\star$, and $C_2$ denote the resulting constants. Set $\zeta = c_1^\circ/2$, and let
$\eta = \eta(k,\zeta)>0$ be the constant appearing in the statement of  Theorem~\ref{thm:morris}. Choose the constant $c$ appearing in the statement of Theorem~\ref{thm:rainbow-factors} sufficiently small so as to satisfy  $c \leq 1$ as well as 
\begin{align} \label{eq:factor-c-one}
 \max \left\{\frac{5c}{\rho_\star^{k-2}}, 
 \frac{2C_2(\sqrt{c_2^\circ}+1)c}{\sqrt\alpha\,\rho_\star^{k-1}} \right\}
 \leq \frac{\eta}{2}.
\end{align}

Let $\cA$ be the family of edge-sets of all $K_k$-factors in $G$; note that $\cA$ is an $h$-edge configuration family on $V(G)$, where $h = \frac{k-1}{2} n$. Continuing our preparations for an application of Theorem~\ref{thm:transfer}, set $S = \cK$ which in turn renders its retention probability to be
$$
 \rho
 =
 \frac{Q-\frac{k-1}{2}n+1}{Q}
 \geq
 \rho_\star.
$$

With Theorem~\ref{thm:morris} serving as the deterministic black-box of the intended application of Theorem~\ref{thm:transfer}, set 
\begin{equation}
 b_\star = \eta(\rho p)^{k-1}n.
 \label{eq:factor-target-beta}
\end{equation}
We prove that $\widehat \beta \leq b_\star$, where $\widehat \beta$ is the degraded bijumbledness parameter appearing in Theorem~\ref{thm:transfer}. Note first that
\begin{equation}
 5\rho\beta
 \leq
 \frac{5c}{\rho^{k-2}}
 (\rho p)^{k-1}n
 \leq
 \frac{\eta}{2} (\rho p)^{k-1}n = b_\star/2
 \label{eq:factor-main-error}
\end{equation}
holds by the premise of Theorem~\ref{thm:rainbow-factors}, by~\eqref{eq:factor-c-one}, and by~\eqref{eq:factor-target-beta}. It thus remains to prove that 
\begin{align} \label{eq::secondSummandInBeta}
C_2 \left(\sqrt{\rho c_2^\circ p n} + \sqrt{\log n} \right) \leq b_\star/2.
\end{align}
Suppose first that $p > 1/2$. Then, $b_{\star} = \eta (\rho p)^{k-1} n = \Omega(n)$, whereas
$$
 \sqrt{\rho c_2^\circ pn} + \sqrt{\log n} = O(\sqrt n), 
$$
implying~\eqref{eq::secondSummandInBeta}. Suppose then that $p \leq 1/2$. Observation~\ref{obs:density-floor} coupled with the assumed upper bound on $\beta$ imposed in the premise of Theorem~\ref{thm:rainbow-factors} imply 
$$
 \frac{\sqrt\alpha}{2} \sqrt{pn} \leq c p^{k-1} n.
$$
Consequently,
\begin{equation}
 \sqrt{pn}
 \leq
 \frac{2c}{\sqrt\alpha\,\rho_\star^{k-1}}
 (\rho p)^{k-1}n.
 \label{eq:factor-sqrt-bound}
\end{equation}
The density consequence following Observation~\ref{obs:density-floor} shows that $pn/\log n \to \infty$; in particular, $\sqrt{\log n} \leq \sqrt{pn}$ holds for all sufficiently large $n$. Equations~\eqref{eq:factor-c-one} and~\eqref{eq:factor-sqrt-bound} then yield
$$
 C_2 \left(\sqrt{\rho c_2^\circ pn} + \sqrt{\log n}\right)
 \leq \frac{\eta}{2} (\rho p)^{k-1} n = b_{\star}/2,
$$
establishing~\eqref{eq::secondSummandInBeta}.

Every graph $J$, targeted by the deterministic black-box of the intended  application of Theorem~\ref{thm:transfer} is $(\rho p,b_\star)$-bijumbled and it satisfies
$$
 \delta(J) \geq
 \frac{c_1^\circ}{2} \rho p n
 = \zeta \rho p n.
$$
By Theorem~\ref{thm:morris}, any such graph $J$ admits a 
$K_k$-factor. Therefore, all conditions listed in Theorem~\ref{thm:transfer} are met. Finally, $\rho pn/\log n \to \infty$, and thus the failure probability stated in~\eqref{eq:transfer-conclusion} is $o(1)$. We conclude that an application of Theorem~\ref{thm:transfer} with the stated parameters a.a.s.\  yields a rainbow $K_k$-factor in $G$.
\end{proof}

\subsection{Rainbow perfect matchings}
\label{sec:perfect-matchings} 

In this section, we prove the branch of Theorem~\ref{thm:rainbow-factors} pertaining to perfect matchings. The first step in this proof is to apply Lemma~\ref{lem:balanced-cut} (stated below) to the $(p,\beta)$-bijumbled host and consequently pass to a $(p,\beta)$-bijumbled spanning balanced bipartite graph.  

\begin{lemma}
\label{lem:balanced-cut}
Let $G$ be a graph on an even number $n$ of vertices.  If
\begin{equation} \label{eq::probBound41}
 n \cdot \exp\left(-\frac{\delta(G)}{36}\right) < 1,
\end{equation}
then there exists an equipartition $V(G) = U \mathbin{\dot\cup} W$ such that $\deg_G(u, W) \geq \deg_G(u)/3$ for every $u \in U$ and $\deg_G(w, U) \geq \deg_G(w)/3$ for every $w \in W$. 
\end{lemma}

The proof of Lemma~\ref{lem:balanced-cut} is postponed until Appendix~\ref{app:balanced-cut}.

The next lemma, whose proof is also delegated to Appendix~\ref{app:deterministic-hall}, serves as the deterministic black-box for the intended application of Theorem~\ref{thm:transfer}. Its origins can be traced back  to~\cite{KrivelevichSudakov2006}.

\begin{lemma}
\label{lem:deterministic-hall}
For every $\lambda>0$, there exists $\zeta := \zeta(\lambda) > 0$ such that the following holds. Let $J = (U,W;E)$ be a bipartite graph satisfying $|U| = |W| = m$, and let $D = p m$.
If $\delta(J) \geq \lambda D$ and $J$ is $(p,\beta)$-bijumbled for $\beta \leq \zeta D$, then $J$ admits a perfect matching.
\end{lemma}

We are now ready to prove the $k=2$ branch of Theorem~\ref{thm:rainbow-factors}. 

\begin{proof}[Proof of Theorem~\ref{thm:rainbow-factors} for $k=2$]
Set $d = p n$, where $n$ is assumed to be even.

Aiming to apply Theorem~\ref{thm:transfer} with $a=1$, $K_0 = 2$, and $\alpha$ as per Theorem~\ref{thm:rainbow-factors}, let  $c_1^\circ, c_2^\circ, d_\star$, and $C_2$ denote the resulting constants. Set $\lambda = c_1^\circ/3$, and let $\zeta := \zeta(\lambda)>0$ be the constant appearing in the statement of Lemma~\ref{lem:deterministic-hall}. Choose the constant $c$ appearing in the statement of Theorem~\ref{thm:rainbow-factors} such that
\begin{equation}
 c \leq \min \left\{1, \frac{\zeta}{20}\right\}.
 \label{eq:matching-choice-c}
\end{equation}

Let $\cA$ be the family of edge-sets of all perfect matchings of $G$;
note that $\cA$ is an $h$-edge configuration family on $V(G)$, where $h = n/2$. Continuing our preparations for an application of Theorem~\ref{thm:transfer}, set $S = \cK$ which in turn renders its retention probability to be
\begin{equation}
 \rho = \frac{Q- n/2 + 1}{Q} \geq \frac{2\eps}{1+2\eps} > 0. 
 \label{eq:matching-rho}
\end{equation}

With Lemma~\ref{lem:deterministic-hall} serving as the deterministic black-box of the intended application of Theorem~\ref{thm:transfer}, 
set 
\begin{equation}
 b_\star = \frac{\zeta}{2} \rho d.
 \label{eq:matching-target-beta}
\end{equation}
We prove that $\widehat \beta \leq b_\star$, where $\widehat \beta$ is the degraded bijumbledness parameter appearing in Theorem~\ref{thm:transfer}. Note first that
\begin{equation}
 5\rho\beta
 \leq
 \frac\zeta4\rho d =
 \frac{b_\star}{2}
 \label{eq:matching-main-error}
\end{equation}
holds by~\eqref{eq:matching-choice-c} and by~\eqref{eq:matching-target-beta}, and since $\beta \leq c d$ by the premise of Theorem~\ref{thm:rainbow-factors}.

Moreover, $d/\log n \to \infty$ and $\rho$ being a positive constant jointly imply that
$$
 C_2\left(
  \sqrt{\rho c_2^\circ d}+\sqrt{\log n}
 \right)
 = o(\rho d).
$$
Consequently, for all sufficiently large $n$,
\begin{equation}
 C_2\left(
  \sqrt{\rho c_2^\circ d}+\sqrt{\log n}
 \right)
 \leq
 \frac{\zeta}{4} \rho d
 =
 \frac{b_\star}{2}
 \label{eq:matching-sampling-error}
\end{equation}
holds for sufficiently large $n$. Equations~\eqref{eq:matching-main-error} and~\eqref{eq:matching-sampling-error} then confirm our claim that 
$\widehat\beta \leq b_\star$.

Let $J$ be any $(\rho p, b_\star)$-bijumbled graph with vertex-set $V(G)$,  satisfying 
\begin{equation}
 \delta(J)\geq\frac{c_1^\circ}{2} \rho d.
 \label{eq:matching-comparison-degree}
\end{equation}
It follows by~\eqref{eq:matching-comparison-degree}, and since $d/\log n \to \infty$, that
$$
 n\cdot \exp\left(-\frac{\delta(J)}{36}\right)
 \leq
 n\cdot \exp\left(-\frac{c_1^\circ\rho d}{72}\right)
 =o(1).
$$
Therefore, by Lemma~\ref{lem:balanced-cut} there exists a partition $V(J) = U \mathbin{\dot\cup} W$ such that $|U| = |W| = n/2$ and every vertex retains at least one third of its degree across the cut. Set
$$
 B = J[U,W],
 \qquad
 m = \frac{n}{2},
 \qquad \textrm{and} \qquad
 D = \rho p m = \frac{\rho d}{2}.
$$
Observe that $B$ is $(\rho p, b_\star)$-bijumbled (in the bipartite sense). Furthermore,
$$
 \delta(B) \geq \frac{1}{3} \delta(J) \geq \frac{c_1^\circ}{6} \rho d
 = \lambda D
$$
holds by~\eqref{eq:matching-comparison-degree}.

It then follows by~\eqref{eq:matching-target-beta} and by Lemma~\ref{lem:deterministic-hall} that $B$, and thus also $J$, admits a perfect matching.

Finally, $d \geq d_\star$ holds for sufficiently large $n$, and
$c \leq 1$ ensures that $\beta \leq d$. Hence, all the conditions of Theorem~\ref{thm:transfer} hold. Moreover, $\rho d/\log n \to \infty$, and thus the failure probability stated in~\eqref{eq:transfer-conclusion} is $o(1)$. We conclude that an application of Theorem~\ref{thm:transfer} with the stated parameters a.a.s.\  yields a rainbow perfect matching in $G$. 
\end{proof}

\section{Rainbow spanning trees in uniformly coloured pseudorandom graphs}
\label{sec:rainbow-trees}

\begin{proof}[Proof of Theorem~\ref{thm:rainbow-trees}]
Fix $\Delta \geq 2$, and set
$$
 L_n = \Delta^{5\sqrt{\log n}}
 \qquad \textrm{and} \qquad
 \rho_\star = \frac{\eps}{1+\eps}.
$$
Aiming to apply Theorem~\ref{thm:transfer} with $a=1$, $K_0 = 2$, and $\alpha$ as per Theorem~\ref{thm:rainbow-trees}, let  $c_1^\circ, c_2^\circ, d_\star$, and $C_2$ denote the resulting constants. Choose $c>0$ sufficiently
small so that 
\begin{equation}
 c \leq 1/10
 \qquad \textrm{and} \qquad
 \frac{C_2(\sqrt{c_2^\circ}+1)c}{2\sqrt\alpha}
 \leq
 \frac{\rho_\star}{8}
 \label{eq:tree-c-choice}
\end{equation}
all hold. 

For an arbitrary prescribed tree $T$ with vertex-set $V(G)$ let $\cA_T$ be the family of edge-sets of all copies of $T$ in $G$; note that $\cA_T$ is an $h$-edge configuration family on $V(G)$, where $h = n-1$. Continuing our preparations for an application of Theorem~\ref{thm:transfer}, set $S = \cK$ which in turn renders its retention probability to be
\begin{equation}
 \rho = \frac{Q-n+2}{Q} \geq \rho_\star.
 \label{eq:tree-rho}
\end{equation}
With Theorem~\ref{thm:han-yang} serving as the deterministic black-box of the intended application of Theorem~\ref{thm:transfer}, set
\begin{equation}
 b_\star = \frac{\rho pn}{4L_n}.
 \label{eq:tree-target-beta}
\end{equation}

We prove that $\widehat \beta \leq b_\star$, where $\widehat \beta$ is the degraded bijumbledness parameter appearing in Theorem~\ref{thm:transfer}. Note first that
$$
 5 \rho \beta \leq 5 c \frac{\rho pn}{4L_n} \leq \frac{b_\star}{2},
$$
where the last inequality holds by~\eqref{eq:tree-c-choice}.

Next, if $p > 1/2$, then
$$
 \sqrt{pn} + \sqrt{\log n} = o\!\left(b_{\star}\right).
$$
Suppose then that $p \leq 1/2$. Observation~\ref{obs:density-floor} coupled with the upper bound on $\beta$ appearing in the statement of Theorem~\ref{thm:rainbow-trees} yield 
\begin{equation}
 \sqrt{pn}
 \leq
 \frac{c}{2\sqrt\alpha}\frac{pn}{L_n}.
 \label{eq:tree-sqrt-bound}
\end{equation}
This then forces $p n/\log n \to \infty$ implying that
$\sqrt{\log n}\leq\sqrt{pn}$ holds for all sufficiently large $n$.
Owing to~\eqref{eq:tree-c-choice}, ~\eqref{eq:tree-rho}, and~\eqref{eq:tree-sqrt-bound}, we conclude that
$$
 C_2\left(
  \sqrt{\rho c_2^\circ pn}+\sqrt{\log n}
 \right)
 \leq b_\star/2,
$$
verifying our claim that $\widehat\beta \leq b_\star$.

Every graph $J$, targeted by the deterministic black-box of the intended  application of Theorem~\ref{thm:transfer} is $(\rho p,b_\star)$-bijumbled and it satisfies
$$
 \delta(J) \geq \frac{c_1^\circ}{2} \rho p n \geq \frac{2\rho pn}{\sqrt{L_n}} = 4\sqrt{b_\star(\rho p)n}, 
$$
where the second inequality holds for sufficiently large $n$ since $L_n \to \infty$.   

Recalling~\eqref{eq:tree-target-beta}, it then follows by 
Theorem~\ref{thm:han-yang} that $J$ is $\cT(n,\Delta)$-universal, and in particular, contains a copy of $T$. 

Finally, since $pn/\log n \to \infty$ and $\rho$ is a positive constant, it follows that the failure probability stated in~\eqref{eq:transfer-conclusion} is $o(1)$. We conclude that an application of Theorem~\ref{thm:transfer} with the stated parameters a.a.s.\  yields a rainbow copy of $T$ in $G$. The arbitrariness of $T$ concludes the proof of Theorem~\ref{thm:rainbow-trees}. 
\end{proof}

\section{Rainbow Hamilton cycles in uniformly coloured pseudorandom graphs}
\label{sec:rainbow-hamilton}

In this section, we prove Theorem~\ref{thm:rainbow-hamilton}. We begin by stating the following lemma, proof of which is delegated to Appendix~\ref{app:deterministic-expansion}; it translates bijumbledness and minimum degree conditions into $C$-expansion. This lemma coupled with Theorem~\ref{thm:ham-expander} form the deterministic black-box required by the palette-transference principle in this application. 

\begin{lemma}
\label{lem:deterministic-expansion}
Fix $a > 0$ and an integer $C \geq 2$, and set
\begin{equation}
 \eta =
 \min\left\{
  \frac{a}{4(1+\sqrt C)},
  \frac18\sqrt{\frac{a}{C+1}},
  \frac1{4C}
 \right\}.
 \label{eq:expansion-eta}
\end{equation}
Let $H$ be an $n$-vertex $(q, \gamma)$-bijumbled graph, where $0 < q := q(n) \leq 1$ and $\gamma \leq \eta q n$. If $\delta(H) \geq a q n$, then $H$
is a $C$-expander.
\end{lemma}

\begin{proof}[Proof of Theorem~\ref{thm:rainbow-hamilton}]
Choose an integer $C \geq 2$ that is also sufficiently large for the assertion of Theorem~\ref{thm:ham-expander} to hold, and set
$$
 \rho_\star = \frac{\eps}{1+\eps}.
$$
Aiming to apply Theorem~\ref{thm:transfer} with $a=1$, $K_0 = 2$, and $\alpha$ as per Theorem~\ref{thm:rainbow-hamilton}, let  $c_1^\circ, c_2^\circ, d_\star$, and $C_2$ denote the resulting constants.
Set $a_\circ = c_1^\circ/2$ and let $\eta := \eta(a_\circ, C) > 0$ be the constant defined in~\eqref{eq:expansion-eta}. Choose $c > 0$ such that
\begin{equation}
 c \leq \min\left\{1, \frac{\eta}{10} \right\}.
 \label{eq:ham-c-choice}
\end{equation}
Finally, choose $K$ to be sufficiently large so as to ensure that
\begin{equation}
 \frac{c_1^\circ \rho_\star K}{8} > 2.
 \label{eq:ham-density-choice}
\end{equation}

Let $\cA$ be the family of edge-sets of all Hamilton cycles of $G$; note that $\cA$ is an $h$-edge configuration family on $V(G)$, where $h = n$. Continuing our preparations for an application of Theorem~\ref{thm:transfer}, set $S = \cK$ which in turn renders its retention probability to be
\begin{equation}
 \rho = \frac{Q-n+1}{Q} \geq \rho_\star.
 \label{eq:ham-rho}
\end{equation}
Set
\begin{equation}
 b_\star = \eta \rho p n.
 \label{eq:ham-target-beta}
\end{equation}

We prove that $\widehat \beta \leq b_\star$, where $\widehat \beta$ is the degraded bijumbledness parameter appearing in Theorem~\ref{thm:transfer}. Note first that
$$
 5 \rho \beta \leq 5 c \rho p n \leq \frac{\eta}{2} \rho p n,
$$
where the first inequality holds by the upper bound on $\beta$ appearing in the statement of Theorem~\ref{thm:rainbow-hamilton}, and the second inequality holds by~\eqref{eq:ham-c-choice}.

Moreover, since $\rho \geq \rho_\star$ and $p n \geq K \log n$, it follows that
\begin{align*}
 \frac{C_2\bigl(\sqrt{\rho c_2^\circ pn}+\sqrt{\log n} \bigr)}{\rho pn} \leq
 \frac{C_2\sqrt{c_2^\circ}}{\sqrt{\rho_\star pn}} + \frac{C_2\sqrt{\log n}}{\rho_\star pn} \leq \eta/2,
\end{align*}
where the last inequality holds for sufficiently large $n$. This confirms our claim that $\widehat\beta \leq b_\star$.

Every graph $J$, targeted by the deterministic black-box of the intended  application of Theorem~\ref{thm:transfer} is $(\rho p,b_\star)$-bijumbled and it satisfies
$$
 \delta(J) \geq \frac{c_1^\circ}{2} \rho p n = a_\circ \rho p n.
$$
It then follows by~\eqref{eq:ham-target-beta} and by 
Lemma~\ref{lem:deterministic-expansion} that any such graph $J$ is a
$C$-expander. Therefore, $J$ is Hamiltonian by Theorem~\ref{thm:ham-expander}.

Finally, it follows by~\eqref{eq:ham-density-choice}, \eqref{eq:ham-rho}, and the assumed lower bound $p n \geq K \log n$, that 
$$
 n \cdot \exp\left(-\frac{c_1^\circ\rho pn}{8}\right) = o(1).
$$
Hence, the failure probability stated in~\eqref{eq:transfer-conclusion} is $o(1)$. We conclude that an application of Theorem~\ref{thm:transfer} with the stated parameters a.a.s.\  yields a rainbow Hamilton cycle of $G$. 
\end{proof}

\section{Concluding remarks}
\label{sec:concluding}

We proved a general transference principle for bijumbled graphs whose edges are coloured independently and uniformly at random from a palette whose surplus beyond the size of the configuration sought is $\eps n$, with $\eps >0$ a prescribed constant. The principle converts rainbow embedding results in uniformly coloured bijumbled graphs to their deterministic colourless  embedding counterparts in bijumbled graphs.

In a companion paper~\cite{AHHPTExact}, we address the reduction of the palette surplus. In that work, we show that the coupling lemma, namely Lemma~\ref{lem:coupling}, can be utilised in order to obtain rainbow perfect matchings and rainbow Hamilton cycles in uniformly coloured bijumbled graphs $G$ where the palette surplus is sublinear, and more precisely of order $\log n/p$, provided $\delta(G) = \Omega(pn)$, $pn = \omega(\log n)$, and $\beta = O(pn)$. Rainbow prescribed bounded-degree spanning trees, we can attain with palette surplus $L_{n,\Delta} \log n/p$, but no corresponding result for clique-factors is attained. The proofs of these results in~\cite{AHHPTExact} do not follow any transference principle; in fact, these are highly specialised and only possible due to the fact that perfect matchings, Hamilton cycles, and bounded degree spanning trees are known to require a certain level of expansion to appear. For clique factors, no such links to expansion are known and hence the inability to handle this configuration in this avenue. 

However, we do show in~\cite{AHHPTExact} that if the density of the graph is mildly elevated, then rainbow embeddings of perfect matchings, Hamilton cycles, as well as prescribed bounded-degree spanning trees in uniformly coloured bijumbled graphs with $\beta = O(pn)$ for the former two and $\beta$ as in Theorem~\ref{thm:rainbow-trees} for the latter are possible with no palette surplus at all. For clique factors, our so-called exact palette rainbow embedding result departs from the parameter $\beta$ seen in the result of Morris~\cite{Morris2025} by a polylogarithmic factor. Our exact-palette results in~\cite{AHHPTExact} are proved using a different transference principle established using spread measures.

\section*{AI disclosure} The authors used ChatGPT Plus for transferring some handwritten notes into \LaTeX, for language proof-reading and polishing, and for completing certain standard calculations. The main results and proofs are due to the authors.

\appendix

\bibliographystyle{amsplain}
\bibliography{Configurations_Lit}

\section{Proofs omitted from Section~\ref{sec:perfect-matchings} - Perfect matchings}\label{app:perfect-matchings}

\subsection{Proof of Lemma~\ref{lem:balanced-cut}}\label{app:balanced-cut}

Choose $U$ uniformly at random amongst the $n/2$-subsets of $V(G)$, and set
$W := V(G) \setminus U$. Fix any vertex $v \in V(G)$; conditional on the part containing $v$, let $X_v$ denote its degree to the other part. Then, $X_v$ is distributed hypergeometrically with mean
$$
 \mathbb E[X_v] = \frac{n}{2(n-1)}\deg_G(v) \geq \frac{1}{2} \deg_G(v).
$$
The hypergeometric lower-tail bound \cite{Hoeffding1963} then yields 
$$
 \Pr\!\left[X_v<\frac13\deg_G(v)\right] \leq
 \exp\left(-\frac{\deg_G(v)}{36}\right).
$$
A union-bound argument over $V(G)$, coupled with~\eqref{eq::probBound41}, then establishes that the required partition exists, concluding the proof.

\subsection{Proof of Lemma~\ref{lem:deterministic-hall}}\label{app:deterministic-hall}
Set
$$
 \theta = \min\left\{\frac14,\frac\lambda4\right\}
 \qquad \textrm{and} \qquad
 \zeta = \min\left\{
  \frac\lambda4,
  \frac12\sqrt{\frac\theta2}
 \right\}
$$
and suppose for a contradiction that Hall's condition fails for $J$. By potentially interchanging the roles of $U$ and $W$, a standard argument shows that there are sets $A \subseteq U$ and $B \subseteq W$ such that
\begin{equation}
 |A|=t,
 \qquad
 |B|=t-1,
 \qquad \textrm{and} \qquad
 1 \leq t \leq \frac{m+1}{2},
 \label{eq:det-hall-witness}
\end{equation}
and such that $E_J(A, Z) = \varnothing$, where $Z := W \setminus B$.

If $t \leq \theta m$, then since $J$ is $(p, \beta)$-bijumbled with $\beta \leq \zeta D$ and $\delta(J) \geq \lambda D$, it follows that
\begin{align*}
 e_J(A, Z) &= \sum_{v \in A} \deg_J(v) - e_J(A,B) \geq \lambda D t - p t (t-1 )-\beta \sqrt{t(t-1)}\\
 &\geq (\lambda - \theta - \zeta) D t \geq \frac{\lambda}{2} D t > 0
\end{align*}
holds; contrary to the definition of $A$ and $Z$.

Suppose then that $\theta m < t \leq (m+1)/2$, and set
$$
 x = \frac t m
 \qquad \textrm{and} \qquad
 z = \frac{|Z|}{m} = \frac{m-t+1}{m}.
$$
Observe that $x \geq \theta$ and $z \geq 1/2$, which in turn imply that  $xz \geq \theta/2$. It then follows by the bijumbledness of $J$ that 
\begin{align*}
e_J(A,Z) \geq p|A||Z| - \beta \sqrt{|A||Z|} = D m \sqrt{xz}
 \left(\sqrt{xz} - \frac{\beta}{D} \right) \geq \frac{1}{2} D m xz > 0,
\end{align*}
where the penultimate inequality relies on $\beta/D \leq \zeta \leq \tfrac{1}{2} \sqrt{\theta/2} \leq \tfrac{1}{2} \sqrt{xz}$. This is again a contradiction to our assumption that $E_J(A, Z) = \varnothing$. Hence, Hall's condition holds, implying that $J$ admits a perfect matching. 

\section{Proofs omitted from Section~\ref{sec:rainbow-hamilton} - Hamiltonicity}
\label{app:hamilton}

\subsection{Proof of Lemma~\ref{lem:deterministic-expansion}}\label{app:deterministic-expansion}
We verify the small-set and large-sets clauses of
Definition~\ref{def:C-expander} separately. For the small-set clause, suppose that there exists some set $X\subseteq V(H)$ of size $1 \leq x := |X| < \frac{n}{2C}$ having fewer than $C x$ external neighbours. Let  $\Gamma_H(X) \subseteq Y \subseteq V(H) \setminus X$ be an arbitrary set of size
$Cx$, and let $Z = V(H) \setminus (X \cup Y)$. By definition, we have
\begin{equation}
 e_H(X,Z) = 0.
 \label{eq:expansion-zero-cut}
\end{equation}
We consider two cases, depending on the value of $x$; in both cases we arrive at a contradiction to~\eqref{eq:expansion-zero-cut}. If $x \leq a n/[4(C+1)]$, then
\begin{align*}
 e_H(X,Z) &= \sum_{v\in X}\deg_H(v)-\evec_H(X,X)-e_H(X,Y) \geq
 a q n x - q x^2 - \gamma x - q C x^2 - \gamma \sqrt{C} x \\
 &= \bigl(a q n - q (C+1) x - \gamma (1+\sqrt C)\bigr) x \geq \frac{a}{2} q n x > 0,
\end{align*}
where the first inequality holds by the minimum-degree and the bijumbledness assumptions appearing in the statement of Lemma~\ref{lem:deterministic-expansion}, and the second inequality follows from the first term in~\eqref{eq:expansion-eta} and the assumed upper bound on $x$.

Suppose then that $x > a n/[4(C+1)]$. Since $x < n/(2C)$ and $C \geq 2$, it follows that
$$
 |Z| = n - (C+1)x > \frac{C-1}{2C} n \geq \frac{n}{4}.
$$
Jointly with the assumed lower bound on $x$ this implies
$$
 \sqrt{x|Z|} > \frac{n}{4} \sqrt{\frac{a}{C+1}}.
$$
The second term in~\eqref{eq:expansion-eta} then yields
$$
 \frac{\gamma}{q \sqrt{x|Z|}} \leq \frac{1}{2}.
$$
Using the bijumbledness of $H$ we then conclude that
$$
 e_H(X,Z)
 \geq
 qx|Z|-\gamma\sqrt{x|Z|}
 \geq
 \frac12qx|Z|
 \geq
 \frac18qnx
 >0.
$$
The obtained contradiction implies that $|\Gamma_H(X)| \geq C|X|$, proving Clause~\textup{(a)}.

\medskip

For the large-sets clause, let $A$ and $B$ be disjoint sets of size at
least $n/(2C)$ each. Then,
\begin{align*}
 e_H(A,B)
 &\geq
 q|A||B| - \gamma \sqrt{|A||B|} \geq \sqrt{|A||B|}
 \left(\frac{qn}{2C} - \gamma\right) > 0,
\end{align*}
where the last inequality holds by the upper bound on $\gamma$ appearing in the statement of Lemma~\ref{lem:deterministic-expansion} and the last term in~\eqref{eq:expansion-eta}. Hence, an edge joins every such pair, completing the proof.

\end{document}